\documentclass{amsart}
\usepackage{amsmath}
\usepackage{amssymb}
\usepackage[all]{xy}
\usepackage{comment}
\usepackage{color}

\theoremstyle{plain}
\newtheorem{thm}{Theorem}[section]
\newtheorem{thm*}{Theorem}[section]
\newtheorem{cor}[thm]{Corollary}
\newtheorem{prop}[thm]{Proposition}
\newtheorem{lemma}[thm]{Lemma}

\newtheorem{lemma*}{Lemma}

\theoremstyle{definition}
\newtheorem{defn}[thm]{Definition}
\newtheorem{remark}[thm]{Remark}

\newtheorem*{remark*}{Remark}

\newtheorem{ex}[thm]{Example}
\newtheorem{notation}[thm]{Notation}

\newtheorem{question*}{Question}

\numberwithin{equation}{thm}

\newcommand{\bM}{\mathbb M}

\newcommand{\cN}{\mathcal N}
\newcommand{\cU}{\mathcal U}

\newcommand{\bG}{\mathbb G}

\newcommand{\bA}{\mathbb A}

\newcommand{\bF}{\mathbb F}

\newcommand{\cC}{\mathcal C}

\newcommand{\cE}{\mathcal E}

\newcommand{\fg}{\mathfrak g}
\newcommand{\fh}{\mathfrak h}

\newcommand{\fu}{\mathfrak u}

\newcommand{\p}{\mathfrak p}

\newcommand{\ol}{\overline}
\newcommand{\ul}{\underline}
\def\sl2{\operatorname{SL_{2(2)}}\nolimits}
\def\Ga2{\operatorname{\mathbb G_{\rm a(2)}}\nolimits}

\newcommand{\bN}{\mathbb N}

\newcommand{\bu}{\bullet}

\date\today

\begin{document}

 \title[Filtrations, 1-parameter subgroups, and rational injectivity]{Filtrations, 1-parameter subgroups, and rational injectivity}
 
 \author[ Eric M. Friedlander]
{Eric M. Friedlander$^{*}$} 

\address {Department of Mathematics, University of Southern California,
Los Angeles, CA}
%%% Email address is optional.
\email{ericmf@usc.edu}
\email{eric@math.northwestern.edu}

\thanks{$^{*}$ partially supported by NSA Grant H98230-15-1-0029 }

\subjclass[2010]{20G05, 20C20, 20G10}

\keywords{filtrations, 1-parameter subgroups, coalgebras, rational injectivity}

\begin{abstract}
We investigate rational $G$-modules $M$ for a linear algebraic group $G$ over an algebraically
closed field $k$ of characteristic $p > 0$ using filtrations 
by sub-coalgebras of the coordinate algebra $k[G]$ of $G$.  Even in the special case of the additive
group $\bG_a$, interesting structures and examples are revealed.  The ``degree" filtration we consider
for unipotent algebraic groups leads to a  ``filtration by exponential degree" applicable to
rational $G$ modules for any linear algebraic group $G$ of exponential type; this filtration is defined
in terms of 1-parameter subgroups and is related to support varieties introduced recently by the
author for such rational $G$-modules.   We formulate in terms of this filtration a necessary and
sufficient condition for rational injectivity for rational $G$-modules.  Our investigation 
leads to the consideration of two new classes of rational $G$-modules: those that are ``mock injective" 
and those that are ``mock trivial".
\end{abstract}

\maketitle

%\tableofcontents 

\section{Introduction}

Beginning with the very special case of the additive group $\bG_a$, we consider the filtration
by degree on rational $\bG_a$-modules which enables us to better understand the 
intriguing category $(\bG_a\text{-}Mod)$ of rational $\bG_a$-modules.  This filtration leads 
to a similarly defined filtration by degree on rational $U_N$-modules, where $U_N \subset GL_N$ is the
closed subgroup of strictly upper triangular matrices, and determines a filtration on rational 
$U$-modules for a closed linear subgroup $U \subset U_N$.
We then initiate the study of a less evident filtration on rational $G$-modules for $G$ a linear
algebraic group of exponential type.
% over an algebraically closed field $k$ of characteristic $p > 0$. 
For rational $U_N$-modules for the unipotent algebraic group $U_N$, this filtration 
of exponential degree is a comparable to the more elementary filtration by degree we first consider. 
Throughout, we fix an algebraic closed field $k$ of characteristic $p > 0$ and consider (smooth) linear
algebraic groups over $k$ together with their rational actions on $k$-vector spaces.

In some sense, this paper is a sequel to the author's recent paper \cite{F15} in which a theory
of support varieties $M \mapsto V(G)_M$ was constructed for rational $G$-modules.  The
construction of the filtration by exponential degree $\{ M_{[0]} \subset M_{[1]} \subset \cdots \subset M \}$ 
uses restrictions of $M$ to 1-parameter subgroups $\bG_a \to G$, and thus is based upon actions 
of $G$ on $M$ at $p$-unipotent elements of $G$.   The role of 1-parameter subgroups to study
rational $G$-modules was introduced in \cite{F11}; the property of $p$-unipotent degree introduced
in \cite[2.5]{F11} is the precursor to our filtration by exponential degree.
The origins of this approach to filtrations lie in considerations of support varieties for modules
for infinitesimal group schemes, varieties which are defined in terms of $p$-nilpotent actions on such modules.

The basic theme of this paper is to investigate rational $G$-modules through their
restrictions via 1-parameter subgroups $\bG_a \to G$.  Whereas the support
variety construction $M \mapsto V(G)_M$ is defined in terms of restrictions of $M$ to a $p$-nilpotent
operator associated to each 1-parameter subgroup of $G$, our present approach involves the 
full information of the restriction of $M$ to all 1-parameter subgroups by using filtrations on 
the category of rational $G$-modules.    We give a necessary and sufficient condition for rational 
injectivity of a rational $G$-module, something we have not succeeded in doing using support varieties.  
Moreover, these filtrations enable us to formulate and study the classes of mock injective modules
(those whose restrictions to every Frobenius kernel are injective) and of mock trivial modules
(those for which actions at 1-parameter subgroups are trivial).  These modules are somewhat 
elusive to construct, but can be shown to exist in great numbers and have interesting properties.

Perhaps the groups of most interest are reductive groups, especially simple groups of classical type.
For such groups $G$, it is instructive to compare the approach in this paper and in \cite{F15} with 
traditional considerations of weights for the action of a maximal torus $T \subset G$ on a rational
$G$-module.  Whereas consideration of weights for $T$ are highly suitable in 
classifying irreducible rational $G$-modules, the action at $p$-unipotent elements has the potential
of recognizing extensions of such modules.
Although our filtration by exponential degree involves actions at $p$-unipotent elements of $G$, 
Example \ref{weights} shows that a bound on the $T$-weights for a rational $G$-module $M$ 
determines a bound on the exponential degree of $M$ for $G$ reductive (where $T$ is a maximal
torus for $G$).  We emphasize that the
filtration by exponential degree applies to rational modules for unipotent groups whose rational
modules are not equipped with a torus action.

We sketch the contents of this paper.  We begin with the most elementary example $G = \bG_a$.
Indeed, this effort was in part motivated by the prospect of establishing a ``local criterion" for a rational
$\bG_a$-module $M$ to be rationally injective  using the support variety $V(\bG_a)_M$, or
equivalently using the restrictions of $M$ to all Frobenius kernels $\bG_{a(r)}$ of $\bG_a$.
Proposition \ref{fail} provides counter-examples to our (unwritten) conjecture that rational injectivity
of rational $\bG_a$-modules is detected in this ``local fashion."  The filtration we consider for
rational $\bG_a$-modules arises from a filtration 
of the coordinate algebra $k[\bG_a]$ by sub-coalgebras $k[\bG_a]_{<d}$.  As observed in  
Proposition \ref{prop-Ga},  the category of comodules
for the sub-coalgebra $k[\bG_a]_{< p^r}$ is naturally isomorphic to the category of rational modules for the
 infinitesimal kernels  $\bG_{a(r)}$ of the linear algebraic group $\bG_a$.  This correspondence is
 very special, following from the simple
 observation that the restriction maps $k[\bG_a] \to k[\bG_{a(r)}]$ are split as maps of coalgebras.
 
 In Section 2, we extend our consideration of filtrations to rational $U$-modules for a unipotent 
 algebraic group $U$ equipped with an embedding in some $U_N$.  To investigate some some properties  
 of a rational $U$-module $M$, we find it more useful to consider the submodules $M_{<p^r} \subset M$
 occurring in the filtration of $M$ than to consider restrictions of $M$ to Frobenius kernels $U_{(r)}$.
 For example, Proposition \ref{Ucrit} gives a necessary and sufficient condition for a rational $U$-module
 $M$ to be injective which is formulated in terms of the filtration $\{ M_{< d}, d > 0 \}$ of $M$.
  For non-abelian $U$, the sub-coalgebras $k[U]_{<p^r} \subset k[U]$ which we use to
 define this degree filtration are not well related to the coordinate algebras of infinitesimal kernels
 $U_{(r)}$; nevertheless, Proposition \ref{numerics} provides some comparison of $k[U]_{< p^r}$ and $k[U_{(r)}]$.  
 
 The key construction of this paper is that of the sub-coalgebras $(k[G])_{[d]} \subset k[G]$ in Definition 
 \ref{def:filt} for $G$ a linear algebraic group of exponential type.  For such $G$, we introduce in Definition \ref{ex-filt} 
 the filtration $\{ M_{[d]} \subset M, d \geq 0 \}$ for any rational $G$-module $M$.  
 As shown in Proposition \ref{prop:relate2}, this filtration is equivalent to that provided in Section \ref{sec:two} 
 in the special case $G = U_N, \ N \leq p$; in particular,
 equivalent to the elementary filtration considered for rational $\bG_a$-modules.  In a few examples
 of finite dimensional rational $G$-modules $M$, we find an explicit value for $d$ such that $M = M_{[d]}$.
 Theorem \ref{properties} provides a list of properties for the filtration $\{ M_{[d]}, \ d \geq 0 \}$ of a
 rational $G$ module $M$ with a structure of exponential type.  In particular, this is a filtration by
 rational $G$-submodules of $M$, satisfies various aspects of functoriality,
  and is independent of the structure of exponential type on $G$.
 The filtration is finite for finite dimensional rational modules and has expected functoriality properties.
 Proposition \ref{Gsupp} gives a relation of this filtration to the theory of support varieties for $G$ 
 as formulated in \cite{F15}.  
 
 In Proposition \ref{coMod}, we establish basic properties of  the functors $(-)_{[d]}$ determining our filtration of 
 rational $G$-modules.   Using some of these properties, we give in Proposition \ref{Gcrit}  a necessary and
 sufficient condition for $M$ to be a rationally injective $G$-module in terms of its filtration
 $\{ M_{[d]}, \ d \geq 0 \}$.   Much of Section \ref{sec:four} is devoted to formulating and then investigation
 the classes of ``mock injectives" and ``mock trivials", rational $G$-modules with interesting properties.
 Mock injectives are rational $G$-modules whose restrictions to all Frobenius kernels $G_{(r)}$ are
 injective $G_{(r)}$-modules.  Every mock injective is rationally injective, but somewhat surprisingly there
 are mock injectives which are not injective.  Mock trivials are rational $G$-modules which have various 
 triviality properties, most notably that the restriction of a mock trivial along any 1-parameter
 subgroup of $G$ is trivial.  Both these classes satisfy familiar closure properties.   
 
 Finally, we conclude in Proposition \ref{Groth} with a Grothendieck spectral sequence relating the 
right derived functors $R^t(-)_{d]}$ of the filtration functor $(-)_{d]}$ for a given degree $d$ with the rational cohomology 
of $G$ for any linear algebraic group $G$ of exponential type.

 We thank Jason Fulman, Julia Pevtsova, Paul Sobaje, and Andrei Suslin for conversations related to the contents
 of this paper.

\section{Rational modules for the additive group $\bG_a$ }
\label{sec:one}

We recall that $\bG_a$ (the additive group) has coordinate algebra $k[T]$ 
equipped with the coproduct 
$$\Delta: k[T] \ \to k[T] \otimes k[T], \quad T \mapsto (T\otimes 1 + 1 \otimes T).$$
In particular, this coproduct on $k[T]$ gives $k[T]$ the structure of a rational $\bG_a$-module
(which is rationally injective).
One can view that action as follows:  for every commutative $k$-algebra $R$ and
for every $a \in \bG_a(R) = R$, the action of $a$ on $f(T) \in R[T]$ is given by
$a \circ f(T) \ = \ f(a+T)$.  

The $r$-th Frobenius kernel $\bG_{a(r)}$ of $\bG_a$,
$$i_r: \bG_{a(r)} \ \subset \ \bG_a,$$
is the closed subgroup scheme with coordinate algebra given by 
$i_r^*: k[\bG_a] = k[T] \ \to \ k[T]/T^{p^r} = k[\bG_{a(r)}]$ and group algebra
(i.e., $k$-linear dual of $ k[\bG_{a(r)}]$) denoted by $k\bG_{a(r)}$.  Using the notation
of  \cite{SFB1}, we let
 $v_0,\ldots,v_{p^r-1}$ be the $k$-basis of $k\bG_{a(r)}$ dual to the standard basis 
$\{ T^j, 0 \leq j < p^r \}$ of $k[T]/T^{p^r}$.  Denote $v_{p^s}$ by $u_s$.  
If $j = \sum_{\ell =0}^{r-1} j_\ell p^\ell, \ 0 \leq j_\ell < p$, then 
$$v_j \ = \  \frac{u_0^{j_0} \cdots u_{r-1}^{j_{r-1}}}{j_0! \cdots j_{r-1}!}.$$

\begin{notation} (see \cite{SFB1})
With notation as above,
$$k\bG_{a(r)} \ \simeq \ k[u_0,\ldots,u_{r-1}]/(u_0^p,\ldots,u_{r-1}^p).$$
For any $r,s > 0$,  the quotient map
$$q_{r,s}: k[\bG_{a(r)}] \cong k[T]/T^{p^{r+s}} \to k[\bG_{a(r)}] \cong k[T]/T^{p^r}$$
sending $T$ to $T$ is a Hopf algebra map, whose dual we denote by 
$$i_{r,s}: k\bG_{a(r)} \to k\bG_{a(r+s)}, \quad u_i \mapsto u_i, \ i < r.$$
The colimit 
$$k\bG_a \ \equiv \ \varinjlim_r  k\bG_{a(r)} \ \simeq k[u_0,\ldots, u_n,\ldots]/(u_0^p,\ldots,u_n^p,\ldots),$$
is the group algebra (or hyperalgebra or algebra of distributions at the identity) of $\bG_a$.
\end{notation}

The following evident lemma makes explicit the action of $k\bG_a$ on a rational $G$-module $M$.

\begin{lemma}
Let $M$ be a rational $\bG_a$-module given by the comodule structure
$\Delta_M: M \to M \otimes k[\bG_a]$.  For $\phi \in k\bG_a$, 
\begin{equation}
\label{action}
\phi(m) \ = \ ((1\otimes \phi)\circ \Delta_M)(m).
\end{equation}
Consequently, the action of $v_j \in k\bG_a$ on the rational 
$\bG_a$-module $M$ is determined by the formula 
\begin{equation}
\label{eq:co}
\Delta_M(m) \ =  \ \sum_j v_j(m) \otimes T^j.
\end{equation}
In particular, the action of $v_j$ on $f(T) = \sum_{n\geq 0} a_nT^n \in k[\bG_a] \simeq k[T]$ is given by
\begin{equation}
\label{act-on-f}
 v_j(f(T))  \ = \ \sum_{n \geq j}  a_n {n \choose j} T^{n-j},
 \end{equation}
 since 
 $$\Delta_{k[\bG_a]}(T^n) \ = \ (T\otimes 1 + 1 \otimes T)^n = \ \sum_{j \geq 0} {n\choose j}T^{n-j} \otimes T^j.$$
\end{lemma}

Using (\ref{eq:co}), we immediately identify those $k\bG_a$-modules which arise as  rational $\bG_a$-modules. 

\begin{prop}
\label{rat-cond}
Let $M$ be a $k\bG_a$-module satisfying the following condition: 
\begin{equation}
\label{cond-rat}
\forall \ m \in M, \ \exists \ \text {only finitely many} \ v_j  \  \text{acting  non-trivially on}  \ m.
\end{equation}
Then the $k\bG_a$-module structure on $M$ (i.e., the action of each $v_j \in k\bG_a$ on $M$)
 arises from the rational $\bG_a$-module structure 
$$M \ \to M \otimes k[\bG_a], \quad m \in M \mapsto \sum_j v_j(m) \otimes T^j.$$  
Conversely, any rational $\bG_a$-module satisfies condition (\ref{cond-rat}).
\end{prop}

We make explicit the following useful consequence of Proposition \ref{rat-cond}.

\begin{cor}
Let $M$ be a rational $\bG_a$-module and $S \subset M$ be a subset.  Then the 
rational $\bG_a$-submodule generated by $S$, $\langle \bG_a \cdot S \rangle$, 
is spanned by $\{ v_j(s); s\in S, j \geq 0 \}$.

In particular, the rational $\bG_a$-submodule  $\langle \bG_a \cdot  f(T)\rangle \ \subset \ k[\bG_a] \simeq k[T]$  
generated by $f(T)$ is
the subspace of $k[T]$ spanned by $\{ v_j(f(T)) \}$ as given in (\ref{act-on-f}).
\end{cor}

\begin{proof}
The span of $\{ v_j(s); s\in S, j \geq 0 \}$ is clearly a $k\bG_a$-submodule of $M$.
Thus, the corollary follows from Proposition \ref{rat-cond}.
\end{proof}

Using a theorem of E. Kummer \cite{Ku} (see also \cite {Ho}), we obtain the following 
explicit description of the rational $\bG_a$-submodule $\langle \bG_a \cdot T^n \rangle$.

\begin{prop}
\label{carries}
The rational $\bG_a$-submodule $\langle \bG_a \cdot T^n \rangle \ \subset k[T]$ has a 
$k$-basis consisting of  those $T^m$ such that adding $n-m$ to $m$ involves no carries
in base-$p$ arithmetic.   In other words, if we write $n$ in base-$p$ as $\sum_{i \geq 0}n_ip^i$
with $0 \leq n_i < p$ for all $i$, then $\langle \bG_a \cdot T^n \rangle$ is spanned by
those $T^m$ for which $m = \sum_{i \geq 0} m_ip^i$ with $m_i \leq n_i$.
\end{prop}

\begin{proof}
By (\ref{act-on-f}), $v_j(T^n)$ is a non-zero multiple of $T^{n-j}$ if and only if $p$ does not 
divide ${n \choose j}$.  Kummer's theorem asserts that the maximal $p$-th
power dividing ${n \choose j}$ equals the number of carries in base-$p$ arithmetic arising in
adding $j$ to $n-j$.
\end{proof}

\begin{ex}
We can easily construct  many non-isomorphic rational $\bG_a$-module structures on 
the underlying vector space of $k[T]$.
Namely, for each $i \geq 0$, choose $n_i \geq 0$ with $\varinjlim_i n_i = \infty$ and choose
$g_i(\ul u) \in k\bG_a \simeq k[u_0,u_1,\ldots]/(u_0^p,u_1^p, \ldots)$ 
such that $g_i(\ul u)$ is a polynomial in the $u_j$'s with $j \geq n_i$.   The 
$k\bG_a $-module structure on $k[T]$ given by setting the action of $u_i$ on $k[T]$ to 
be that of $g_i(\ul u)$ on $k[\bG_a] \simeq k[T]$
with its structure arising from the coproduct of $\bG_a$ is a rational $\bG_a$-module 
by Proposition \ref{rat-cond} (since only finite many $g_i(\ul u)$'s act non-trivially on a given $T^n$
by (\ref{act-on-f}) ).
\end{ex}

\begin{remark}
Different choices of $g_i(\ul u)$ in the preceding example can lead to isomorphic rational $\bG_a$-modules.
For example, let $\theta: \bN \to \bN$ be a bijection and set $g_i(\ul u) = u_{\theta(i)}$.  The resulting
module structure on $k[T]$ is isomorphic to that of $k[\bG_a]$  via the $k$-linear isomorphism 
$\Theta_\Sigma: k[T] \to k[T]$ sending the monomial $\prod_{i>0}(T^{p^i})^{d_i}$ to $\prod_{i>0}(T^{p^{\theta(i)})^{d_i}}$
where $0 \leq d_i < p-1$
\end{remark}

The following elementary proposition justifies the functor $(-)_{<d}$ of Definition \ref{endo}.

\begin{prop}
\label{prop:phi}
For any rational $\bG_a$-module $M$ and any $\phi \in k\bG_a$,
$$M_\phi \ \equiv \ \{ m \in M: \phi(m) = 0 \} \ \subset \ M$$
is a rational $\bG_a$-submodule of $M$.   
Moreover, if $f: M \to N$ is a map
of rational $\bG_a$-modules, then $f$ restricts to $M_\phi \to N_\phi$.
\end{prop}

\begin{proof}
To show that $M_\phi \subset M$ is a rational $\bG_a$-submodule it suffices
by Proposition \ref{rat-cond} to show that $\psi \cdot M_\phi \subset M_\phi$ for
any $\psi \in k\bG_a$.  This follows immediately from the commutativity of $\bG_a$
(implying the commutativity of $k\bG_a$).

The second assertion concerning a map $f: M \to N$ of rational $\bG_a$-modules follows
from the fact that $f$ necessarily commutes with the action of $\phi$.
\end{proof}

Proposition \ref{prop:phi} enables the formulation of many natural filtrations on $(\bG_a-Mod)$.
The motivation for considering the following is given by Proposition \ref{Ga-retract}.
 
\begin{defn}
\label{endo}
For any $d \geq 1$, we define the idempotent endo-functor
$$(-)_{< d}: (G{\text -}Mod) \ \to \ (G{\text -}Mod), \quad
M \ \mapsto \ M_{< d} \equiv  \{ m\in M: v_j(m) = 0, j \geq d \}.$$
In other words, $m \in M_{ <d}$ if and only if $\Delta(m) \in M\otimes k[T]_{<d}$.

For any rational $\bG_a$-module $M$, we consider the {\it degree filtration}
$$M_{< 1} \ \subset M_{<2 } \ \subset  M_{<3}  \ \subset \ \cdots  \ \subset  M$$ 
of $M$ by rational $\bG_a$-submodules.
\end{defn}

The following proposition, established in \cite{F15},  follows easily from the 
observation that the coproduct $\Delta_M: M \to M\otimes k[\bG_a]$ defining the
rational $\bG_a$-module structure on $M$ sends a finite dimensional subspace of
$M$ to a finite dimensional subspace of $M \to M\otimes k[\bG_a]$.

\begin{prop} \cite[2.6]{F15}
Each finite dimensional rational $\bG_a$-module lies in the image of $(-)_{< d}$
for $d$ sufficiently large.  Consequently, 
\begin{enumerate}
\item For any rational $\bG_a$-module $M$, \ $M = \cup_d M_{< d}$.
\item  If $M$ is finite dimensional, then $M \ = \ M_{<d}$ for $d >>0$.
\end{enumerate}
\end{prop}

Unlike for other linear groups considered in later sections, the filtration on 
the coordinate algebra $[\bG_a] = k[T]$ of $\bG_a$ can be viewed as a coalgebra splitting
of restriction maps $k[\bG_a] \ \to \ k[\bG_{(r)}]$ as observed in the next proposition.

\begin{prop}
\label{Ga-retract}
For each $d > 0$, the rational submodule $j_d: k[T]_{< d} \subset k[T]$ is a
sub-coalgebra.  

Moreover,
$$pr_{p^r} \circ j_d: k[T]_{< p^r} \subset k[T] \to k[T]/T^{p^r}, \ r > 0$$
is an isomorphism of coalgebras.
\end{prop}

\begin{proof}
The fact that $j_d$ is an embedding of coalgebras follows from the form
of the coproduct $\Delta: k[T] \to K[T] \otimes k[T]$ which sends $T^n$
to $(T\otimes 1 + 1 \otimes T)^n$; thus $\Delta$ applied to a polynomial $f(T)$
of degree $< d$ is mapped to $\sum_i g_i \otimes h_i$ with the degree of 
each $g_i$ and each $h_i$ $< d$.

The fact that $pr_d \circ j_d$  is injective (and thus an isomorphism by dimension
considerations) is evident by inspection.  For $d = p^r$, one easily checks that
the coalgebra structure on $k[T]$ induces a coalgebra structure on $k[T]/T^{p^r}$.
\end{proof}

We summarize some of the relationships between various functors on rational $\bG_a$-modules.
We denote the abelian category of such rational modules either by $(\bG_a\mbox{-}Mod)$ or by
$(k[\bG_a]\mbox{-}coMod)$; we denote the category of rational modules for the infinitesimal group scheme
$\bG_{a(r)}$ either by $(\bG_{a(r)}\mbox{-}Mod)$ or by $(k[\bG_{a(r)}]\mbox{-}coMod)$.

We denote by 
$$\rho_r: (\bG_a{\text -}Mod) \  \to \ (\bG_{a(r)}\mbox{-}Mod)$$
the restriction functor sending a rational $\bG_a$-module $M$ with coporoduct $M \to M\otimes k[\bG_a]$
to the comodule for $k[\bG_{a(r)}]$ with coproduct defined by composition with the projection
$pr_{p^r}: k[\bG_a] \to k[\bG_{a(r)}]$.

\begin{prop}
\label{prop-Ga}
Consider the full subcategory $\iota_d: (k[\bG_a]_{<d}{\text -}coMod)\ \subset \ 
(k[\bG_a]{\text -}coMod)$ of rational $\bG_a$-modules
$M$ whose coproduct is of the form $M \to M \otimes k[\bG_a]_{<d}$.
\begin{enumerate}
\item
The image of $\iota_d$ consists of rational $\bG_a$-modules $M$ such that  $M = M_{< d}$. 
\item
For any $d > 0$, $\iota_d$ is left adjoint to functor $(k[\bG_a]{\text -}coMod) \to (k[\bG_a]_{<d}{\text -}coMod)$
given by $(-)_{<d}$.
\item
For any $r > 0$, the composition
$$\rho_r \circ \iota_{<p^r}: (\bG_a{\text -}Mod)_{<p^r} \ \stackrel{\sim}{\to} \ Mod(\bG_{a(r)})$$
is an equivalence of categories.
\end{enumerate}
\end{prop}

\begin{proof}
The first statement is essentially a tautology. 
The fact that $\iota_d$ is left adjoint to $(-)_{<d}$ follows from the observation that if
$f: M \to N$ is a map of $k\bG_a$-modules and if $\phi \in k\bG_a$ vanishes on $M$,
then $f$ factors (uniquely) through $N_\phi \subset N$.

The last statement is a consequence of the isomorphism 
$pr_{p^r} \circ j_{p^r}: k[T]_{< p^r} \stackrel{\sim}{\to} k[T]/T^{p^r}$ of 
Proposition \ref{Ga-retract}.  Namely, viewing $\rho_r $ and $\iota_{<p^r}$ as functors 
on categories of comodules, $\iota_{< p^r}$ is determined on comodules
by composing with $j_{p^r}$ and $\rho_{p^r}$ is determined by composing with $pr_{p^r}$.
\end{proof}

%%%%%%%%%%%%%%%%%%%%%%%%%%%
%%%%%%%%%%%%%%%%%%%%%%%%%%%%

\section{Rational modules for unipotent groups}
\label{sec:two}

Let $U_N$ denote the unipotent radical of the standard (upper triangular)
Borel subgroup of the general linear group $GL_N$.  Then $k[U_N]$ is a polynomial
algebra on the strictly upper triangular coordinate functions $\{ x_{i,j}; \ 1 \leq i < j \leq N \}$.
We equip $k[U_N]$   with the grading determined by
setting the degree of each $x_{i,j}$ equal to 1.  A closed embedding $U \subset U_N$
of linear algebraic groups is said to be {\it linear} if the ideal $I_U \subset k[U_N]$ defining 
$U \subset U_N$ is generated by functions $f(x_{i,j})$ of degree 1.  This implies that the 
maximal ideal at the identity of $U$,  $\frak m_U \subset k[U]$, is generated by $m$
elements where $m = dim(U)$, so that $k[U]$ can be identified with the symmetric algebra
$S^\bu(\frak m/\frak m_U^2)$.

For any $d > 0$, we set $k[U_N]_{<d} \subset 
k[U_N]$ equal to the subspace of polynomials (functions on $U_N$) of degree $< d$.

\begin{prop}
\label{unip}
Let $i: U \to U_N$ be a closed embedding of linear algebraic groups.  Set \
$k[U]_{< d} \ \subset \ k[U]$ \ equal to the image under $i^*$
of \ $k[U_N]_{ < d} \ \subset \ k[U_N]$.   
The map of Hopf algebras $i^*: k[U_N] \to k[U]$ induces for each $d >0$ a map of coalgebras
$$k[U_N]_{< d} \to k[U]_{< d}.$$
\end{prop}

\begin{proof}
The coproduct $\Delta_{U_N}: k[U_N] \to k[U_N] \otimes k[U_N]$  of the Hopf algebra $k[U_N]$ is
a map of algebras, determined by 
$$\Delta_{U_N}(x_{i,j}) \ = \ (x_{i,j}\otimes 1)+(\sum_{t; i < t < j} x_{i,t} \otimes x_{t,j})+(1\otimes x_{i,j}).
$$
In particular, if $f \in k[U_N]$ has degree $< d$ and if 
$\Delta_{U_N}(f) = \sum_i f_i^\prime \otimes f_i^{\prime \prime}$, then each $f_i^\prime$ and each
$f_i^{\prime \prime}$ also has degree $< d$.

Because $i^*$ is a map of Hopf algebras, $i^*$ determines a commutative square of algebras
\begin{equation}
\label{hopf}
\xymatrix{k[U_N] \ar[r]^-{\Delta_{U_N}} \ar[d]_{i^*} & k[U_N] \otimes k[U_N]  \ar[d]^{i^* \otimes i^*}\\
k[U] \ar[r]_-{\Delta_U} & k[U] \otimes k[U] } .
\end{equation}
A simple diagram chase implies that (\ref{hopf}) restricts to a commutative square
\begin{equation}
\xymatrix{k[U_N]_{<d} \ar[r]^-{\Delta_{U_N}} \ar[d]_{i^*} & k[U_N]_{<d}  \otimes k[U_N]_{<d}   \ar[d]^{i^* \otimes i^*}\\
k[U]_{<d}  \ar[r]_-{\Delta_U} & k[U] _{<d} \otimes k[U]_{<d} } .
\end{equation}
\end{proof}

In particular, each subspace $ k[U]_{< d} \ \subset \ k[U]$ is a rational $U$-module with
coaction $k[U]_{< d}  \ \to \ k[U]_{< d}  \otimes k[U]$ given by the coalgebra structure on $k[U]_{< d} $.

\begin{defn}
\label{Ufilt}
Let $U$ be a linear algebraic group provided with a closed embedding $i: U \to U_N$ for some $N$.
For any rational $U$-module $M$ and any $d > 0$, we define 
\begin{equation}
\label{Md}
M_{< d} \ \equiv \ \{ m \in M: \Delta_M(m) \in M \otimes k[U]_{< d} \}.
\end{equation}
The {\it degree filtration} on $M$ is the filtration 
$$M_{< 1} \ \subset \ M_{< 2} \ \subset M_{< 3} \ \subset \cdots M.$$
If $M = M_{<d}$, then we say that $M$ has filtration degree $< d$.
\end{defn}

The following proposition will prove useful at many points; in particular, it implies that the degree
filtration of (\ref{Md}) is a filtration by rational $\bG_a$-modules.

\begin{prop}
\label{coideal}
Let $C$ be a coalgebra over $k$ and $i: B \subset C$ a right coideal (i.e., $\Delta_C: C \to C\otimes C$
restricts to $\Delta_B: B \to B\otimes C$).  For any right $C$-comodule $M$ (i.e., $\Delta_M: M \to M\otimes C$),
the subspace 
$$M^\prime \ \equiv \ \Delta^{-1}_M(M \otimes B) \ \subset \ M$$
is a right $C$-subcomodule of $M$.  Moreover, if $i: B \ \subset \ C$ is a sub-coalgebra,
then $M^\prime$ is a right $B$-comodule.

In particular, let $G$ be a linear algebraic group and let $B \subset k[G]$ be a right co-ideal (i.e., a rational
$G$-submodule of $k[G]$).  Then for any rational $G$-module $M$,
the subspace  $M^\prime \equiv \Delta^{-1}_M(M \otimes B) \subset M$ is a rational $G$-submodule.
\end{prop}

\begin{proof}
The comodule structure map $\Delta_M: M \to M\otimes C$ for $M$ is a map of right $C$-comodules
provided that the right $C$-comodule structure on $M\otimes C$ is given by sending $m\otimes c$
to $m\otimes \Delta_C(c)$.   Since $i: B \subset C$ is a right coideal, $1\otimes i: M \otimes B \subset M\otimes C$
is a right $C$-comodule.   We claim that the pre-image $M^\prime \equiv  \Delta^{-1}_M(M \otimes B)$ 
(in the abelian category of right $C$-modules) of the right $C$-subcomodule $M \otimes B \subset M\otimes C$ 
with respect to the map $\Delta_M: M \to M \otimes C$ of right $C$-comodules is a right $C$-comodule as 
asserted.  Namely, the kernel $K \subset M \oplus (M\otimes B) $ of the map of right $C$-modules
$\Delta_M - (1\otimes i): M \oplus (M\otimes B) \ \to \ M\otimes C$ maps isomorphically via projection onto the first 
summand of $M \oplus (M\otimes B)$ to $M^\prime \subset M$ since $1\otimes i: M\otimes B \to M\otimes C$ is injective.
Furthermore, the right $C$-coproduct $\Delta_{M^\prime}: M^\prime \to M^\prime \otimes C$
(the restriction of $\Delta_M$) has image in $M\otimes B$ by definition of $M^\prime$.

If $B \ \subset \ C$ is a sub-coalgebra, then the right $C$-comodule structure on $M \otimes B$,
$\Delta_{M\otimes B}: M \otimes B \ \to \ M \otimes B \otimes C$, is a right $B$-comodule structure
and thus restricts to a right $B$-comodule structure on $M^\prime$.

Specializing the previous argument to $C = k[G]$, we get the second assertion concerning rational
$G$-modules.
\end{proof}

\begin{remark}
To understand the statement of Proposition \ref{coideal}, it may be useful to consider the special
 case in which $G$ is a discrete group
acting on a $k$-vector space $M$, and $B \subset C$ is taken to be the inclusion of group algebras
$k[G/H] \subset k[G]$ for some normal subgroup $H \subset G$.  In this case, $M^\prime \subset M$ is the
subspace of elements $m^\prime \in M$ with the following property:  if $gm = m^\prime$ for some 
$g \in G, m \in M$, then $ghm = m^\prime$ for every $h \in H$. 
\end{remark}

Specializing $C$ to $k[U]$ and $B$ to $k[U]_{< d}$ in Propostion \ref{coideal}, we conclude the following.

\begin{prop}
\label{prop:Md}
Let $U$ be a linear algebraic group provided with a closed embedding into some $U_N$ and let
$M$ be a rational $U$-module.  
\begin{enumerate}
\item
For any $d > 0$, the subspace  $M_{< d} \ \subset \ M$ of a rational 
$U$-module $M$ is a  rational $U$-submodule of $M$ whose structure
arises from a comodule structure for the sub-coalgebra $k[U]_{< d}$ of $k[U]$.
\item
Conversely, if $N \subset M$ is a rational $U$-module whose structure 
arises from  a comodule structure for $k[U]_{< d}$, then $N = N_{< d}$.
\item
In particular, the degree filtration $\{ M_{< d}, d > 0 \}$ is a filtration of $M$ by rational $U$-submodules.
\end{enumerate}
\end{prop}

\begin{ex}
Let $M$ be a rational $G$-module, with $G = U_2 \subset GL_2$ (isomorphic to $\bG_a$).  
Then the degree filtration on $M$ as formulated in  Definition \ref{Ufilt} equals that in Definition \ref{endo}
for $M$ viewed as a $\bG_a$-module).
\end{ex}

\begin{ex}
\label{Upoly}
Let $M$ be a polynomial $GL_N$-module, homogeneous of degree $d-1$ (i.e., a comodule for
$k[\bM_N]_{d-1} \subset k[GL_N]$), 
and consider $M$ via restriction as a rational $U_N$-module.  Then $M$ has filtration
degree $< d$.  This follows immediately by observing that restriction of $M$ to $U_N$ has
coproduct $M \to M\otimes k[U_N]$ equal to the composition 
$(1_M \otimes \pi )\circ \Delta_M: M \to M \otimes k[\bM_N] \to M\otimes k[U_N]$,
where $\pi: k[\bM_N] \to k[U_N]$ sends $x_{i,j}$ with $i < j$ to $x_{i,j}$; $x_{i,i}$ to 1; and 
$x_{i,j}$ with $i > j$ to 0.
\end{ex}

The following proposition asserts that $(-)_{<d}$ is an idempotent functor determining a right adjoint to 
the embedding $(k[U]_{< d}{\text -}coMod) \ \hookrightarrow \ (k[U]{\text -}coMod)$.  This is a 
generalization of Proposition \ref{prop-Ga}.1 and itself is 
generalized in Proposition \ref{coMod}.

\begin{prop}
\label{iota}
Let $U$ be a linear algebraic group provided with a closed embedding 
$U \to U_N$ for some $N$.  Then for any $d > 0$ and any rational $U$-module $M$
\begin{enumerate}
\item 
$M_{<d} \ = \ (M_{<d})_{<d}$ (where $M_{<d}$ is given in (\ref{Md}));
\item
the natural embedding given by the inclusion of coalgebras $k[U]_{< d} \to k[U]$,
$$\iota_d: (k[U]_{< d}{\text -}coMod) \ \hookrightarrow \ (k[U]{\text -}coMod),$$
is left adjoint to the functor
$$(-)_{<d}: (k[U]{\text -}coMod) \ \to \ (k[U]_{< d}{\text -}coMod), \quad M \mapsto M_{< d}.$$
\end{enumerate}
\end{prop}

\begin{proof}
By Proposition \ref{coideal}, $\Delta_{M_{<d}}: M_{<d} \to M_{<d} \otimes k[U]$ has image in
$M_{<d} \otimes k[U]_{<d}$.  Thus, 
$$(M_{<d})_{<d} \ = \ (\Delta_{M_{<d}})^{-1}(M_{<d} \otimes k[U]_{<d}) \ = \ M_{<d}.$$

If $M = M_{<d}$ and $N$ are rational $U$-modules, then any map $f: M \to N$ of 
rational $U$-modules fits in a commutative square
\begin{equation}
\xymatrix{
M \ar[r]^-{\Delta_M} \ar[d]_f & M\otimes k[U_N]_{<d}   \ar[d]^{f \otimes i_d}\\
N  \ar[r]_-{\Delta_N} & N \otimes k[U]. } 
\end{equation}
Consequently, $f$ factors uniquely through $N_{<d}$; this means that $(-)_{<d}$ is right
adjoint to $\iota_d$.
\end{proof}

The fact that $(-)_{<d}$ admits an exact left adjoint immediately implies that it
sends injectives to injectives as we make explicit in the following corollary of
Proposition \ref{iota}.

\begin{cor}
\label{inj-restr}
For any rationally injective $U$-module $L$ and any $d > 0$,
$L_{< d}$ is an injective object of $(k[U]_{< d}{\text -}coMod)$.
\end{cor}

For any linear algebraic group $G$ with coordinate algebra $k[G]$,
we denote by $G^{(r)}$ the linear algebraic group whose coordinate
algebra $k[G^{(r)}]$ is the base change $k\otimes_k k[G]$ along
the $p^r$-th power map $k \to k$.  The $r$-th Frobenius map
is the natural map $F^r: k[G^{(r)}] \to k[G]$ of $k$-algebras sending 
$1\otimes f$ to $f^p$; for $G$ defined over $\bF_p$, we may view this 
as an endomorphism of $k[G]$.  We define the $r$-th Frobenius kernel
$$G_{(r)} \ \equiv ker\{ F^r: G \to G^{(r)} \} \subset \ G,$$
so that $k[G_{(r)}] \ = \ k[G] \otimes_{k[G^{(r)}] } k$, where
$k[G^{(r)}]  \to k$ is the counit of $k[G^{(r)}]$.  Thus,  we may identify $k[G_{(r)}]$ with
$k[G]/(f^{p^r},f \in \frak m_G)$, the quotient of $k[G]$ by the ideal generated by 
$p^r$-th powers of elements in the maximal ideal at the identity
(i.e., generated by $f^{p^r}$ for all $f\in k[G]$ with $f(1) = 0$). 
The quotient map $k[G] \twoheadrightarrow  k[G_{(r)}]$ is a map of Hopf algebras.

The following proposition should be compared and contrasted with Proposition \ref{Ga-retract}.

\begin{prop}
\label{numerics}
Let $U$ be a connected linear algebraic group  provided with a closed embedding 
$i: U \to U_N$ for some $N$.  Let $m$ denote the dimension of $U$.  Then for any $r > 0$ 
\begin{enumerate}
\item
The composition $k[U]_{< p^r} \subset \ k[U] \ \twoheadrightarrow \ k[U_{(r)}]$ is injective.
\item  
The dimension of $k[U_{(r)}]$ equals $p^{rm}$.
\item
The composition 
$$k[U]_{< p^r \cdot \frac{N(N-1)}{2}} \ \subset \ k[U] \ \twoheadrightarrow \ k[U_{(r)}]$$ is surjective.
\item  
If the closed embedding $i: U \subset U_N$ is linear, then
the dimension of $k[U]_{< p^r}$ equals ${\frac{m(m-1}{2}+p^r}\choose {p^r}$; this is never a $p$-th
power, and  is not divisible by $p$ if $m < p^r$.
\end{enumerate}
\end{prop}

\begin{proof}
Every non-zero element of $\frak m_U$ has positive degree in $k[U]$ (i.e., is not constant),
so that $k[U]_{< p^r} \cap \frak m_U^{p^r} \ = \ 0$; this proves (1).

The computation in (2) of the dimension of $k[U_{(r)}]$ can be verified as follows.  For $r=1$,
$k[U_{(1)}]$ is dual to the restricted enveloping algebra of $\fu = Lie (U)$ and therefore 
has dimension equal to $p^m$.  Furthermore, the quotient $U_{(r)}/U_{(r-1)}$ is isomorphic
to $U_{(1)}$ for $r > 1$, so the computation is concluded using induction on $r$.

The commutativity of the following diagram with surjective vertical maps
\begin{equation}
\xymatrix{k[U_N]_{< p^{r+s}} \ar[r] \ar[d]  & k[U_N] \ar[r] \ar[d]  & \ol k[U_{N(r)}] \ar[d]\\
k[U]_{< p^{r+s}} \ar[r]   & k[U] \ar[r]   & \ol k[U_{(r)}] 
}
\end{equation}
reduces the proof of (3) to the case that $U = U_N$.  In this case, $m = \frac{N(N-1)}{2}$.
We view $k[U_N]_{<p^r\cdot m} \to k[U_{N(r)}]$
as the surjective map from the space of polynomials in $m$ variables spanned by monomials
of total degree $<p^r\cdot m$ to the space of polynomials in the same variables spanned
by monomials not divisible by the $p^r$-th power of any variable.  

To prove (4), observe that the dimension of those polynomial functions of total
degree $< p^r$ in $m$ variables equals the dimension of those homogeneous polynomial functions
of degree $p^r$ in $m+1$ variables.  One checks recursively that the latter dimension
equals ${m+p^r}\choose {p^r}$ which equals $((m+p^r)\cdot \cdots \cdot (1+p^r))/m!$.  Clearly, this is
not divisible by $p$ if $m < p^r$ and is never a $p$-th power (see Kummer's Theorem utilized in
the proof of Proposition \ref{carries}).
\end{proof}

For some time, we tried to prove the following injectivity criterion for rational $\bG_a$-modules:
if $M$ is a rational $\bG_a$, then $M$ is injective.
As the following proposition makes clear, this ``support variety criterion for
injectivity" fails miserably not just for $\bG_a$ but for any connected unipotent 
algebraic group.

\begin{prop}
\label{fail}
Let $U$ be a connected unipotent algebraic group of positive dimension 
and let $U \subset G$ be a closed embedding of $U$ in a reductive group $G$.  
Then $k[G]$ is not injective as a rational $U$-module, whereas the restriction of 
$k[G]$ to each Frobenius kernel $U_{(r)}$ of $U$ is injective.
\end{prop}

\begin{proof}
As shown in \cite[2.1,4.5]{CPS77} (see also Proposition \ref{induction} below), 
for a closed subgroup $H$ of a reductive algebraic group
$G$ to satisfy the condition that $k[G]$ is injective as a rational $k[H]$-module, it is 
necessary and sufficient that $H$ be reductive.  In particular, $k[G]$ is not injective 
as a rational $U$-module.

On the other hand, $k[G]$ is rationally injective as a $G$-module and  thus projective 
as an  $kG_{(r)}$-module for any $r > 0$.    Since
$kG_{(r)}$ is free as a $kU_{(r)}$-module (see, for example, \cite{J}), we conclude
that $k[G]$ is projective (in fact, free) when restricted to each Frobenius kernel $U_{(r)}$.
\end{proof}

In order to provide a necessary and sufficient criterion of for rational injectivity 
for rational $U$-modules, we mention the following structure property for rationally
injective modules for a unipotent algebraic group $U$.
We remind the reader $V\otimes k[G]$ is rationally injective for any affine algebraic 
group $G$ and any $k$-vector space $V$ (see \cite{J}).

\begin{prop}
\label{H0}
Let $U$ be a connected, unipotent algebraic group.
Let $L$ be a rationally injective $U$-module and set $L_0 = H^0(U,L)$.   
There exists a map $f: L \to L_0 \otimes k[U]$ of rational $U$-modules whose restriction to $L_0$
identifies $L_0$ with $L_0 \otimes 1 \subset L_0  \otimes k[U]$;
in particular, such a map $f$ is injective when restricted to $L_0 \subset L$.

Moreover, a map 
$f: L \ \to \ L_0 \otimes k[U]$ of rational $U$-modules is an isomorphism if 
and only if the induced map $H^0(f): H^0(U,L) = L_0 \  \to \ H^0(U,L_0 \otimes k[U])$  is 
an isomorphism of vector spaces.
\end{prop}

\begin{proof}
The existence of a map $f: L \to L_0 \otimes k[U]$ of rational $U$-modules which ``is the
identity on $L_0$" is an immediate consequence of the extension mapping property of 
the rationally injective $k[U]$-module $L_0 \otimes k[U]$ applied to the monomorphism $L_0 \subset L$.
This map is necessarily injective because it is injective on the socle of $L$.

Clearly, the condition that $H^0(f)$ be an isomorphism is necessary for $f$ to be an
isomorphism.
Since $U$ is unipotent, $H^0(U,L)$ is the socle of both $L$ and $L_0 \otimes k[U]$.  
Thus, the condition that $H^0(f)$ be injective implies that $f$ is itself injective because
a non-trivial kernel of $f$ would have to meet the socle of $L$ non-trivially.  If $f$
is injective, the rational injectivity of $L$ implies the existence of some $g: L_0 \otimes k[U] \ \to \ L$
with $g \circ f = id$.  On the other hand, if $L_0 \otimes k[U] \ \simeq L \oplus L^\prime$
with $L^\prime \not= 0$, then $H^0(f)$ is not surjective.
\end{proof}

If $U$ is a unipotent algebraic group, rational injectivity of a rational $U$-module $L$ 
can be detected on submodules $L_{< p^r}$ by arguing by induction on the dimension
of the socle of $L$.  This is done in the following proposition, which motivates the 
criterion of Proposition \ref{Gcrit} for an arbitrary linear algebraic group of exponential type.

\begin{prop}
\label{Ucrit}
Let $U$ be a linear algebraic group provided with a closed embedding $i: U \ \to \ U_N$
for some $N$.  Then a rational $U$-module $L$ is rationally injective if and only if 
$L_{< p^r} \subset L$ is injective as a  $k[U]_{<p^r}$-comodule for each $r > 0$.

In particular, if $U \simeq \bG_a$, then a rational $\bG_a$-module $L$ is rationally 
injective if and only if for all $r > 0$ the restriction of  $L_{< p^r} \subset L$ to $k\bG_{a(r)}$ 
(via $\rho_r \circ \iota_{<p^r}$ of Proposition \ref{prop-Ga}) is free.
\end{prop}

\begin{proof}
By Corollary \ref{inj-restr}, the restriction of a rationally injective $U$-module is injective
as a $k[U]_{<p^r}$-comodule for each $r > 0$.  For $U \simeq \bG_a$, Proposition 
\ref{prop-Ga}(3) tells us that injective $k[\bG_a]_{<p^r}$-comodules are injective (equivalently,
free) $k\bG_{a(r)}$-modules.

To prove the converse, consider some rational $U$-module $L$ such that 
$L_{< p^r} \subset L$ is injective as a  $k[U]_{<p^r}$-comodule for each $r > 0$.  
Let $L_0$ denote the socle of $L$.   As in Proposition \ref{H0}, let 
$f: L \ \to \ L_0\otimes k[U]$ be some injective map of rational $U$-modules extending $L_0 \subset L_0\otimes k[U]$.
We proceed to show that $f$ is an isomorphism.  Namely, for each $r > 0$, we use the injectivity of
$L_{<p^r}$ as a $k[U]_{<p^r}$-comodule to extend the identity map $L_{<p^r} \to L_{< p^r}$ along the
monomorphism $f_r: L_{< p^r} \to L_0 \otimes k[U]_{< p^r}$ to some $k[U]_{<p^r}$-comodule 
homomorphism $g_r: L_0 \otimes k[U]_{< p^r} \to L_{< p^r}$.  The fact that $g_r$ extends the 
identity map implies that $g_r$ is surjective.  On the other hand, $g_r$ induces an isomorphism
on socles, so must be injective.   Thus, each $f_r$ is an isomorphism (with inverse $g_r$).
Since $f = \varinjlim_r f_r$, we conclude that $f$ is an isomorphism.
\end{proof}

\begin{remark}
Proposition \ref{Ucrit} stands in contrast with Proposition \ref{numerics}: $k[U]_{<p^r}$
is not free as a $U_{(r)}$-module if the dimension of $U$ is greater than 1.
\end{remark}

%%%%%%%%%%%%%%%%%%%%%%%%%%%%%%%%
%%%%%%%%%%%%%Sec 3%%%%%%%%%%%%%%%%%%%

\section{Filtrations on rational $G$-modules for $G$ of exponential type}
\label{sec:three}

Throughout this section, $G$ denotes a connected linear algebraic group
with Lie algebra $\fg$.  We denote by $\cC_r(\cN_p(\fg))$ the variety of $r$-tuples 
$(B_0,\ldots,B_{r-1})$ of pair-wise commuting, $p$-nilpotent elements of $\fg$;
in other words, each $B_i$ satisfies $B_i^{[p] } = 0$ and each pair $B_i, B_j$ satisfies
$[B_i,B_j] = 0$.  We denote by $V_r(G)$ the variety of height $r$ 
infinitesimal 1-parameter subgroups  $\bG_{a(r)} \to G$ of $G$.

We begin by recalling from \cite[1.6]{F15} the definition of a structure of exponential type 
on a linear algebraic group, a definition which  extends the formulation in \cite{SFB1} 
of an embedding $G \subset GL_n$ of exponential type.
Up to isomorphism (as made explicit in \cite[1.7]{F15}), if such a structure exists then it is unique.

\begin{defn}
\label{str-exptype}
\cite[1.6]{F15}
Let $G$ be a linear algebraic group with Lie algebra $\fg$.   A structure of exponential type
on $G$ is a morphism of $k$-schemes
\begin{equation}
\label{Exp}
\cE: \cN_p(\fg) \times \bG_a \ \to G, \quad (B,s) \mapsto \cE_B(s)
\end{equation}
such that
\begin{enumerate}
\item
For each $B\in \cN_p(\fg)(k)$, $\cE_B: \bG_a \to G$ is a 1-parameter subgroup.
\item
For any pair of  commuting $p$-nilpotent elements $B, B^\prime \in \fg$,
the maps $\cE_B, \cE_{B^\prime}: \bG_a \to G$ commute.
\item
For any commutative $k$-algebra $A$, any $\alpha \in A$,  and any 
$s \in \bG_a(A)$, \ $\cE_{\alpha \cdot B}(s) = \cE_B(\alpha\cdot s)$.
\item  Every 1-parameter subgroup $\psi: \bG_a \to G$ is of the form 
$$ \cE_{\ul B} \ \equiv \ \prod_{s=0}^{r-1} (\cE_{B_s} \circ F^s)$$
for some $r > 0$, some $\ul B \in \cC_r(\cN_p(\fg))$; furthermore, $\cC_r(\cN_p(\fg)) \to
V_r(G), \ \ul B \mapsto \cE_{\ul B} \circ i_r$ is an isomorphism for each $r > 0$.
\end{enumerate}

A connected linear algebraic group which admits a structure of exponential type is
said to be a {\bf linear algebraic group of exponential type}.

Moreover,  $H \subset G$ is said to be an embedding of {\it exponential type} if
$H$ is equipped with the structure of exponential type given by restricting that
provided to $G$; in particular, we require $\cE: \cN_p(\fg) \times \bG_a \ \to \ G$
to restrict to $\cE: \cN_p(\fh) \times \bG_a \ \to \ H$.
\end{defn}

The following explicit example \cite[1.2]{SFB1} helps to illuminate Definition \ref{str-exptype}.

\begin{ex}
Let $\fg l_N = Lie(GL_N)$.  Then the pairing 
\begin{equation}
\label{expGLN}
\cE:  \cN_p(\fg l_N) \times \bG_a \ \to \ GL_N, \quad (B,t) \mapsto exp_B(t) \equiv 
1+tX+\frac{(tX)^2}{2}+ \cdots + \frac{(tX)^{p-1}}{(p-1)!}
\end{equation}
defines a structure of exponential type on $GL_N$.
\end{ex}

Many familiar linear algebraic groups are linear algebraic groups of exponential 
type as recalled in the following example.

\begin{ex}
The following linear algebraic groups are of exponential type.
\begin{itemize}
\item
Any simple algebraic group of classical type, any parabolic subgroup of such a group,
any unipotent radical of such a parabolic subgroup is of exponential type \cite[1.8]{SFB1}.
\item
Any reductive algebraic group $G$  with Coxeter number $h(G)\leq p$ and $PSL_p$ not
a factor of $[G,G]$ is of exponential type \cite{Sobj1}.
\item
Any unipotent radical
of a parabolic subgroup of $SL_N$ or a product of commuting root groups in $SL_n$.
\item
Any unipotent algebraic group $U$ of nilpotent
class $< p$ is of exponential type, with the Campbell-Hausdorff-Baker formula
determining the exponential structure.
\end{itemize}
\end{ex}

We  introduce a filtration $\{ M_{[d]}, d > 0 \}$ on a rational $G$ module $M$ for $G$ a 
linear algebraic group of exponential type.  For $G = U_N$ with $N \leq p$, 
 this filtration 
is comparable to the filtration of Definition \ref{Ufilt} as seen in Remark \ref{rem:compare};
this filtration is a natural extension of the condition that a rational $G$ module have
exponential degree $< p^r$ as introduced in \cite[4.5]{F15}.

We first define $\{ k[G]_{[d]}, d > 0 \}$ on the coordinate algebra $k[G]$ of a 
linear algebraic group $G$ of exponential type; this is shown in Proposition \ref{subcoalg}
to be a filtration by sub-coalgebras.  

\begin{defn}
\label{def:filt}
Let $G$ be a linear algebraic group  of exponential type, \ 
$\cE: \cN_p(\fg) \times \bG_a \ \to \ G$.    For any $d \geq 0$, we define $(k[G])_{[d]} \subset k[G]$
as follows:
$$(k[G])_{[d]} \ \equiv \ \{ f \in k[G]: (\cE_{B*}(v_j))(f) = 0, \ \forall B\in \cN_p(\fg), j > d \}.$$
In other words, $(k[G])_{[d]}$ is the pre-image under $\cE^*: k[G] \to k[\bG_a] \otimes k[\cN_p] $ 
of $(k[\cN_p])[T]_{< d+1}$.
\end{defn}

In what follows, we shall employ  the notation
$$\cE_B^*: k[G] \to k[\bG_a], \quad \cE_{B*}: k\bG_a \to kG$$
 for the maps on coordinate algebras and group algebras induced by $\cE_B \bG_a \to G$.

\begin{prop}
\label{subcoalg}
Let $G$ be a linear algebraic group  of exponential type.
For any $d \geq  0$, $(k[G])_{[d]} \ \subset \ k[G]$ is a sub-coalgebra.  In particular, 
$(k[G])_{[d]}$ is a rational $G$-submodule of $k[G]$.

Moreover, $(k[G])_{[0]}$ is a sub-Hopf algebra of $k[G]$.
\end{prop}

\begin{proof}
Let $f \in k[T], \ B \in \cN_p(\fg)$. 
Because $\cE_B: \bG_a \to G$ is a morphism of algebraic groups, 
$$\cE_B^*(\Delta_G(f)) \ = \ \Delta_{\bG_a}(\cE_B^*(f)) \ \in \ k[\bG_a] \otimes k[\bG_a].$$
On the other hand, if $p(T) = \cE_B^*(f) \in k[\bG_a]$ has degree $< d+1$, then
$\Delta_{\bG_a}(p(T) )$ is of the form $ \sum_i p_i(T) \otimes p_i^\prime(T)$ with each $p_i(T), \ p_i^\prime(T)$
having degree $< d+1$.  Thus, the coproduct of $k[G]$ restricts to a coproduct on $(k[G])_{[d]}.$ 

The multiplicative structure of the commutative $k$-algebra $k[G]$ restricts to a 
multiplicative structure $(k[G])_{[0]}\otimes (k[G])_{[0]} \to (k[G])_{[0]}$ (since the product of two
polynomials in $k[T]$ of degree $< 1$ is again of degree $< 1$), thereby verifying that $(k[G])_{[0]}$
is a sub-Hopf algebra of $k[G]$.
\end{proof}

We investigate the relationship between the filtration of Definition \ref{def:filt} to that of Proposition \ref{unip}.
We point out the following elementary computation: for $B \in U_N$, $exp_B(x_{i,j}) \in k[T]$
equals the $(i,j)$-th entry of the matrix $1 + BT + B^2{\cdot}T^2/s + \cdots  B^{p-1}T^{p-1}/(p-1)!$.

\begin{prop}
\label{prop:relate1}
Let $U$ be a linear algebraic group provided with a closed embedding 
 $i: U \subset U_N$ of exponential type.   For any $d > 0$, 
 $$k[U]_{<d} \ \subset \ (k[U])_{[(p-1)(d-1]}.$$
 
 On the other hand, if $ k \not= k[U]_{[0]} \subset k[U]$ (e.g., $U = U_N$ with $N > p$; see 
 Example \ref{notcompare}), then $k[U]_{[0]}$ is
 not contained in $k[U]_{<d}$ for any $d$. 
\end{prop}

\begin{proof}
In the special case $U= U_N$, $\cE_B^*(x_{i,j}), 1 \leq i < j \leq N$, 
is the polynomial in $T$ whose coefficient of $T^n$ is  $1/(n!)$ times the
$(i,j)$-th entry of $B^n$ for any $n, \ 1 \leq n < p$.   Thus, the ring homomorphism 
$\cE_B^*$ sends a polynomial in the $x_{i,j}$ 
of degree $\leq d-1$ to a polynomial in $T$ of degree $\leq (p-1)(d-1)$.  Hence $\cE_{B*}(v_j)$ applied to 
$f \in k[U_N]_{<d}$ is 0 for $j > (p-1)(d-1)$.
 For $U \subset U_N$ of exponential type, this argument restricts to $U$.  This establishes
 the  inclusion $ k[U]_{<d} \subset \ (k[U])_{[(p-1)(d-1]}.$
 
 If there exists a non-constant function $f \in k[U]_{[0]}$, then $\tilde f \in k[U_N]$ mapping to $f$ must
 have positive degree.   Thus, powers of $f$ (also in the Hopf algebra $k[U]_{[0]}$) have arbitrarily
 large degree.
 \end{proof}
 
 The following examples point out that even for $U_N$ the comparison of the filtrations
of Definition \ref{def:filt} and \ref{unip} is not entirely straight-forward.

\begin{ex}
\label{notcompare}
Let $U = U_3$ and consider $f = 2x_{1,3} - x_{1,2}x_{1,3} \in k[U_3]$.  Then for all $B \in U_3$,
the degree of $exp_B^*(f) \in k[T]$ is $\leq 1$ whereas $f \notin k[U_3]_{<2}$.

Let $U = U_N$ with $N >p$.  Then $f = \prod_{i = 1}^{N-1} x_{i,i+1}$ satisfies the property that 
$exp_B^*(f) = \prod_{i=1}^{N-1} exp_B^*(x_{i,i+1}) = 0$ for any $B$ such that $B^p$ = 0.
Therefore $f \in (k[U_N])_{0]}$.
\end{ex}

In the special case $U = U_N, \ N \leq p$, we verify that our two filtrations are equivalent.

\begin{prop}
\label{prop:relate2}
Assume that $N \leq p$.  Then
$$(k[U_N])_{[e-1]} \ \subset \ k[U_N]_{<d} \ \subset \ (k[U_N])_{[(p-1)(d-1]}$$
provided that $e(N-1) < d$.
\end{prop}

\begin{proof}
To prove the inclusion $(k[U_N])_{[e]} \ \subset \ k[U_N]_{<d}$ for $N \leq p$,  
for each $f \in k[U_N]$ of degree $D \geq d$ in the matrix functions $x_{i,j}$ (with $i < j$) 
we must exhibit a strictly
upper triangular matrix $B$ such that $exp_B^*(f) \in k[T]$ has degree $\geq e$.  We write
$$f \ = \ \sum_{\ul d} a_{\ul d} x_{\ul d}, \quad x_{\ul d} = \prod_{i<j} x_{i,j}^{d_{i,j}}, \quad deg(\ul d) = \sum d_{i,j}.$$
We say that a monomial $x_{\ul d^\prime}$ is a contraction of another monomial $ x_{\ul d}$ 
if $x_{\ul d^\prime}$ can be obtained from 
$x_{\ul d}$ by iterated replacement of a string of factors of the form 
$x_{i,s_1}, x_{s_1,s_2},\ldots,x_{s_\ell,j}$ by the single factor $x_{i,j}$.  We say that a monomial 
$x_{\ul e}$ appearing in $f$ (i.e., such that  $a_{\ul e} \not= 0$) is reduced for $f$ provided that each
monomial $x_{\ul d}$ of the same degree as $x_{\ul e}$ satisfying 
$d_{i,j} \not=0$ only if $e_{i,j} \not=0$  (but not necessarily appearing in $f$) has no contraction
of smaller degree appearing in $f$.

Starting with a monomial $x_{\ul D}$ appearing of $f$ of top degree $D$, we identify 
by following repeated contractions some reduced monomial $x_{\ul e}$ appearing in $f$
of degree $e$ with $e(N-1) \geq D \geq d$.  We consider matrices $B$ with the property that 
$b_{i,j} \not= 0$  only if $e_{i,j} \not= 0$ for this reduced monomial  $x_{\ul e}$ of degree $e$.   
We claim that the coefficient of $T^e$ in $exp_B^*(f)$ equals the sum
\begin{equation}
\label{reduced}
\sum_{\ul d} a_{\ul d} (\prod_{i<j}b_{i,j}^{d_{i,j}}),
\end{equation}
where the sum is taken over all monomials  $x_{\ul d}$  of $f$ of degree $e$ with 
$d_{i,j} \not= 0$ only if $e_{i,j} \not= 0$.   First, observe that for any monomial $x_{\ul d^\prime}$,
we have $exp_B^*(x_{\ul d^\prime}) = \prod (\exp_B^*(x_{i,j}))^{d_{i,j}^\prime}$.   This implies that if
$x_{\ul d^\prime}$ has degree $> e$, then the coefficient of $T^e$ in $exp_B^*(x_{\ul d^\prime})$ is 0.
Moreover, if $x_{\ul d}$ has degree $= e$, then
the coefficient of $T^e$ in $exp_B^*(x_{\ul d})$ equals 
$a_{\ul d} (\prod_{i<j}b_{i,j}^{d_{i,j}})$.   Finally, if $x_{\ul d^\prime}$ has degree $< e$,
then the coefficient of $T^e$ in $exp_B^*(x_{\ul d^\prime})$ is non-zero only if some factor
$x_{i,j}$ is non-zero on a power $B^s$ of $B$ with $s > 0$; this implies that $x_{\ul d^\prime}$
is a contraction of some monomial $x_{\ul d}$ of degree $e$ with $d_{i,j} \not=0$ only if $e_{i,j} \not=0$; 
since $x_{\ul e}$ is assumed to be reduced, this implies
that $x_{\ul d^\prime}$ does not appear in $f$.

We view (\ref{reduced}) as a polynomial in the variables $b_{i,j}$ with $(i,j)$ running through
pairs $1\leq i < j \leq N$ such that $e_{i,j} \not= 0$ for our chosen $x_{\ul e}$ for the given $f$ of
degree $D$.  Since this polynomial is not constant, we may find values for the $b_{i,j}$'s constituting
a matrix $B$ such that $exp_B^*(f)$ has non-zero coefficient of $T^e$; in other words, 
$f \notin (k[U_N])_{[e-1]}$.

The second inclusion is a special case of Proposition \ref{prop:relate1}.
\end{proof}

One consequence of the following proposition is that $(k[G])_{[d]}$ is never finite dimensional
if $G$ is a non-trivial reductive algebraic group, since for reductive $G$ the kernel ideal
of $k[G] \to k[\cU_p]$ is an infinite dimensional vector space.  Indeed, for $G$ reductive, the 
Krull dimension of $k[G]$ equals the Krull dimension of $k[\cU]$ plus the rank of $G$, and
the Krull dimension of $k[\cU]$ is greater or equal to the Krull dimension of $k[\cU_p]$.  Here,
and in the proposition below, $\cU \subset U$ is the closed subvariety of unipotent elements
and $\cU_p \subset \cU$ is the closed subvariety of elements whose $p$-th power is the identity.

\begin{prop}
\label{prop:embed}
Let $G$ be a linear algebraic group of exponential type and let $\cU_p(G) \ \subset \ G$ denote
the closed variety of $p$-unipotent elements of $G$.  Then
$\cE^*: k[G] \ \to \ k[\cN_p(\fg)] \otimes k[\bG_a]$ factors through an embedding 
\begin{equation}
\label{embed}
\ol {\cE^*}: k[\cU_p(G)] \ \hookrightarrow \  k[\cN_p(\fg)] \otimes k[\bG_a].
\end{equation}
Consequently, the augmentation ideal of $(k[G])_{[0]}$ (i.e., the functions $f \in  (k[G])_{[0]} \subset k[G]$ 
such that $f(id_G) = 0$) equals the ideal in $k[G]$ of those functions on $G$ which vanish on $\cU_p(G) \subset G$.
\end{prop}

\begin{proof}
If $f \in k[G]$ vanishes on $\cU_p(G)$ (i.e. lies in the ideal defining the closed subvariety
$\cU_p(G) \subset G$), then $\cE_B^*(f) = 0$ for all $B \in \cN_p(\fg)$ since $\cE_B:
\bG_a \to G$ factors through $\cU_p$. 
This implies that $\cE^*$ factors through $k[\cU_p(G)]$,
for if not then there exists some $f \in k[G]$ vanishing on $\cU_p(G)$ such that $\cE^*(f) 
= \sum_i f_i \otimes T^i \in k[\cN_p(\fg)]\otimes k[T]$ is non-zero;  for such an $f$ and some $i$
with $f_i \not= 0$, find $B \in \cN_p(\fg)$ such that $f_i(B) \not= 0$; then $\cE_B^*(f) \not= 0$,
contradicting the assumption that  $\cE^*(f) \not= 0$.  On
the other hand, this induced map $k[\cU_p(G)] \to k[\cN_p(\fg)] \otimes k[\bG_a]$ is injective,
for its composition with $id \otimes eval_1: k[\cN_p(\fg)] \otimes k[\bG_a] \to k[\cN_p(\fg)]$
is an isomorphism by condition (4) of Definition \ref{str-exptype}.

In particular, we have shown that the kernel of $\cE^*: k[G] \to  k[\bG_a]\otimes   k[\cN_p(\fg)] $ equals 
the kernel of the restriction map $k[G] \to k[\cU_p]$ which equals the ideal 
of those functions on $G$ which vanish on $\cU_p(G) \subset G$.  On the other hand, 
$(k[G])_{[0]}$ consists of those $f \in k[G]$ such that $\cE_B^*(f)$ is constant (i.e., lies in
$k$ for all $B$).  Therefore, the augmentation ideal of $(k[G])_{[0]}$ equals the ideal of
$\cU_p(G)$.
\end{proof}

We introduce the filtration by exponential degree on a rational $G$-module, 
an ``extension" of the degree filtration on a rational $U$-module given in 
Definition \ref{Ufilt}.  

\begin{defn}
\label{ex-filt}
Let $G$ be a linear algebraic group  of exponential
type and let $M$ be a rational $G$-module.  For any $d \geq 0$, we define
\begin{equation}
\label{Mr}
M_{[d]} \ \equiv \ \{ m\in M:  \ \Delta_M(m) \in M \otimes (k[G])_{[d]}\}.
\end{equation}

The filtration {\it by exponential degree} on $M$ is the filtration 
$$M_{[0]} \ \subset \ M_{[1]} \ \subset \ \cdots \ \subset \ M.$$
\end{defn}

We say that $M$ has {\it exponential degree $\leq d$} if $M \ = \ M_{[d]}$.  

\begin{prop}
\label{alternate}
With notation as in Definition \ref{ex-filt}, $M_{[d]} \ \subset \ M$ consists of 
those elements $m \in M$ such that $(\cE_B)_*(v_j) \in kG$ vanishes on $m$ for all
$B \in \cN_p(\fg)$ and all $j >  d$.
\end{prop}

\begin{proof}
If $\Delta_M(m) \in M \otimes k[G]$ lies in $M \otimes (k[G])_{[d]}$, then the composition 
with $(1\otimes ev_B \otimes 1)\circ \cE^*: M \otimes k[G] \to M \otimes k[\cN_p(\fg)] \otimes M\otimes k[T]$ has
image in $M \otimes k[T]_{\leq d}$ and equals $\cE_B^* \circ \Delta_M(m)$; thus,
$(\cE_B)_*(v_j)$ applied to $m$ vanishes for all $B \in \cN_p(\fg)$ and all $j >  d$.

Conversely, if $\Delta_M(m) = \sum_\alpha m_\alpha \otimes f_\alpha \in M \otimes k[G]$ with 
some $f_\alpha$ of degree $> d$, then 
$$\cE^*(\Delta_M(m)) \ = \ \sum_\alpha m_\alpha \otimes \sum_j g_{\alpha,j} \otimes T^j
\in M \otimes k[\cN_p(\fg)] \otimes k[T]$$
with some $g_{\alpha,j} \not= 0$ for $j > d$.  Then for any $B\in \cN_p(\fg)$ such that
$g_{\alpha,j}(B) \not= 0$, $\cE_B^*(v_j)(m) = \sum_{\alpha,j} g_{\alpha,j}(B)m_\alpha \otimes T^j \not= 0$.
\end{proof}

\begin{remark}
\label{rem:compare}
Let $G$ be a linear algebraic group  of exponential
type and let $M$ be a rational $G$-module.  The condition that  $M$ has 
exponential degree $\leq p^r-1$ is equivalent to the condition that $M$
has exponential degree $< p^r$ in the sense of \cite[4.5]{F15}.

Take $G$ to equal $U_N$ for some $N \leq p$ and $M$ be a rational $U$-module.  
By Propositions \ref{prop:Md} and \ref{prop:relate2}, 
the condition $M = M_{<d}$ in the sense of Definition \ref{Ufilt} implies that 
$M = M_{[(p-1)(d-1)]}$ in the sense of Definition \ref{ex-filt}.  Similarly, the condition that
$M = M_{[e-1]}$ in the sense of Definition \ref{ex-filt} implies that $M = M_{<d}$
 in the sense of Definition \ref{Ufilt} provided that $e(N-1) < d$.
\end{remark}

The following are natural examples of rational $G$-modules of
(explicitly) bounded exponential degree.

\begin{ex}
\label{Schur}
Let $S(N,d)$ denote the Schur algebra, so that the linear dual of $S(N,d)$
 is the coalgebra $k[\bM_N]_d$, the vector space of polynomials
homogeneous of degree $d$ in the variables $\{ x_{i,j}, 1 \leq i,j \leq N \}$.  We verify that 
\begin{equation}
\label{schur-deg}
k[\bM_N]_d \ \hookrightarrow \ k[GL_N]_{[(p-1)d]}.
\end{equation}
Namely, if $B$ is a $p$-nilpotent, $N \times N$ matrix, then $exp_B^*(x_{i,j}) \in k[T] = k[\bG_a]$ 
is the $(i,j)$-th entry
of the matrix $1 + BT + B^2T^2/2 + \cdots + B^{p-1}T^{p-1}/(p-1)!$ which has degree $\leq p-1$
(as a polynomial in $T$).
Thus, if $f \in k[GL_N]$ is homogenous of degree $d$ in the $x_{i,j}$ (i.e., in the image of $k[\bM_N]_d$), then
$exp_B^*(f)$ has degree $\leq (p-1)d$.

Consequently, if $M$ is a polynomial representation of $GL_N$ homogeneous of
degree $d$ (i.e., a comodule for $k[\bM_N]_d$), then $M = M_{[(p-1)d]}$.
\end{ex}

\begin{ex}
\label{weights}
Let $G$ be a reductive group with a structure of exponential type and let 
$M$ be a rational $G$-module all of whose high weights $\mu$
satisfy the condition that   \ $2\sum_{j=1}^l \langle \mu,\omega_j^\vee \rangle \ < \  p^r$.
Here, $\{ \omega_1,\ldots,\omega_\ell \}$ is the set of fundamental dominant weights of $G$
(with respect to some $T \subset B \subset G$) and 
$$\omega_j^\vee = 2\omega_j/\langle \alpha_j,\alpha_j \rangle.$$
Then $M = M_{[p^r-1]}$ as seen in \cite[2.7]{F11} following \cite[4.6.2]{CLN}.
\end{ex}

We provide various properties of our filtration by exponential degree of rational $G$-modules.

\begin{thm}
\label{properties}
Let $G$ be a linear algebraic group of exponential
type and let $M$ be a rational $G$-module. 
\begin{enumerate}
\item
The abelian category of comodules for the coalgebra $(k[G])_{[d]}$ equals the
full subcategory of $(G{\text -}Mod)$ consisting of those rational $G$-modules of 
exponential degree $ \leq d$ (i.e., $M$ such that $M = M_{[d]}$).
\item
The filtration $\{ M_{[d]}, \ d  \geq 0 \}$ of $M$ is independent of the choice of structure of exponential
type for $G$.
\item
If $M$ is a finite dimensional rational $G$-module, then $M = M_{[d]}$ for $ d>> 0$.
\item  $M \ = \ \cup_d M_{[d]}$ for any rational $G$-module $M$.
\item
The  filtration of $k[G]$ by exponential degree, $\{ (k[G])_{[d]}, \ d  \geq 0 \}$, is finite if and only if
$\cN_p(\fg) = \{ 0 \}$.
\item
If $M$ has exponential degree $\leq d$, then its Frobenius 
twist $M^{(1)}$ has exponential degree $\leq pd$.
\item
If $M^\prime$ has exponential degree $\leq  d$ and if $f: M^\prime \to M$ is a map of
rational $G$-modules, then $f(M^\prime) \ \subset \ M_{[d]}$.
\item
If $f: M^\prime \hookrightarrow M$ is an inclusion of rational $G$-modules and if 
$m^\prime \in M^\prime \backslash M^\prime_{[d]}$, then $f(m^\prime) \in M \backslash M_{[d]}$.
\item
If $j: H \subset G$ is an embedding of exponential type and if $M$ has exponential 
degree $\leq d$ as a rational $G$-module, then the restriction to $H$ of
$M$ has exponential degree $\leq  d$ as a rational $H$-module.
\end{enumerate}
\end{thm}

\begin{proof}
Property (1) is merely a rephrasing of the condition that a rational $G$-module $M$ satisfies the 
condition $M = M_{[d]}$.

Property (2) follows from  \cite[1.7]{F15}; property (3) is established in  \cite[2.6]{F15}.  Since
any rational $G$-module is a union of its finite dimensional submodules, property (4) follows
from property (3).

If $\cN_p(\fg) = 0$, then by condition (4) of Definition \ref{str-exptype} there are no non-trivial
1-parameter subgroups $\bG_a \to G$ so that $k[G] = k[G]_{[0]}$.  Conversely, 
If $\psi: \bG_a \to G$ is a non-trivial 1-parameter subgroup, then $\psi^*: k[G] \to k[\bG_a]$
has infinite dimensional image so that for each $d > 0$ there exist $f \in k[G]$ which do not 
lie in $k[G]_{[d]}$.  

The Frobenius twist of a rational $G$-module $M$, $M^{(1)}$,
 has as its coproduct structure $\Delta_{M^{(1)}}: M^{(1)} \to M^{(1)} \otimes k[G]$
the composition 
$$(1_M \otimes F) \circ (\Delta_M)^{(1)}: M^{(1)} \to M^{(1)} \otimes k[G]^{(1)} \to M^{(1)} \otimes k[G]$$
where $F: k[G]^{(1)} \to k[G]$ is the $k$-linear map sending
$\alpha\otimes f \in k\otimes_\phi k[G]$ to $\alpha f^p \in k[G]$.
For any 1-parameter subgroup $\psi: \bG_a \to G$, $\psi^*(f^p) = (\psi^*(f))^p \in k[\bG_a]$,
so that the image under $F: k[G]^{(1)} \to k[G]$  of $(k[G])_{[d]})^{(1)}$ lies in $(k[G])_{[pd]}$,
thereby establishing property (6).

Properties (7) and (8) are easy consequences Definition \ref{ex-filt} and
the fact that a map $f: M^\prime \to M$ is a map of $k[G]$-comodules.

Finally, the condition that  $j: H \subset G$ be an embedding of 
exponential type (see Definition \ref{str-exptype})  implies the commutativity of
the square
\begin{equation}
\label{comm}
\xymatrix{
k[G] \ar[d]_{j^*} \ar[r]^-{\cE^*} & k[\cN_p(\fg)] \otimes k[\bG_a] \ar[d]^{j^* \otimes 1} \\
k[H] \ar[r]_-{\cE^*} & k[\cN_p(\fh)] \otimes k[\bG_a]. }
\end{equation}
The surjectivity of $j^*$ together with the commutativity of (\ref{comm}) implies that $j^*$
restricts to $j_d^*: (k[G])_{[d]} \to (k[H])_{[d]}$.  Thus, if the coproduct $\Delta_M: M \to M \otimes k[G]$
for the rational $G$-module $M$ factors through $M \otimes (k[G])_{[d]}$, then the coproduct
for $M$ restricted to $H$ factors through $M \otimes (k[H])_{[d]}$.  
\end{proof}

A key definition of \cite{F15} is that of the ($p$-nilpotent) action at a 1-parameter subgroup 
$\cE_{\ul B}: \bG_a \to G$ of a linear algebraic group $G$ of exponential type acting on a
rational $G$-module $M$.  In \cite[2.6.1]{F15}, this is defined to be the action of 
$\sum_{s \geq 0} (\cE_{B_s})_*(u_s) \ \in \ kG$ acting on $M$.   

\begin{defn}\cite[4.4]{F15}
\label{ratsupp}
Let $G$ of a linear algebraic group $G$ of exponential type and $M$ a rational $G$-module.
Then the {\bf support variety} of $M$, $V(G)_M \subset V(G)$, is the subvariety of those 1-parameter
subgroups  $\cE_{\ul B}: \bG_a \to G$ at which the action of $G$ on $M$ is not free (in the sense
that $M$ is not free $k[t]/t^p$-module with $t$ acting as $\sum_{s \geq 0} (\cE_{B_s})_*(u_s)$).
\end{defn}

For any $\ul B \in \cC_\infty(\cN_p(\fg))$ and any $r > 0$, we set $\Lambda_r(\ul B)$
equal to $(B_{r-1}, B_{r-2}, \ldots,B_0)$.  Thus, $\cE_{\Lambda_r(\ul B)}: \bG_a \to G$
can be viewed as an infinitesimal 1-parameter subgroup $\bG_{a(r)} \to G_{(r)}$.  As shown
in \cite[4.3]{F15}, the $\pi$-points $k[u]/u^p \to kG_{(r)}$ given by sending $u$ to
$\sum_{s=0}^{r-1} (\cE_{B_s})_*(u_s)$ and to $(\cE_{\Lambda_r(\ul B)})_*(u_{r-1})$ are 
equivalent.   One consequence of this equivalence of $\pi$-points is the close 
relationship between support varieties as defined in Definition \ref{ratsupp} for a linear
algebraic group $G$ and the support varieties for the Frobenius kernels $G_{(r)}$ as defined
for any finite group scheme.    This enables the proof (given in \cite[4.6.1]{F15}) of the 
following proposition giving a consequence involving support varieties of the hypothesis
that a rational $G$-modules has exponential degree $< p^r$.

\begin{prop} (\cite[4.6.1]{F15})
\label{Gsupp}
Let $G$ be a linear algebraic group   of exponential type 
and let $M$ be a rational $G$-module such that $M = M_{[d]}$.   If $p^r > d$
(so that $M$ has exponential degree $< p^r$), then
the support variety $V(G)_M$ of $M$ satisfies 
$$V(G)_M \ = \ \Lambda_r^{-1}(V_r(G)_M).$$
Here, $\Lambda_r: \cC_\infty(\cN_p(\fg)) \to \cC_r(\cN_p(\fg))$ sends $\ul B$ to
$(B_{r-1}, B_{r-2}, \ldots,B_0)$. 
\end{prop}

\begin{proof}
This follows immediately from Remark \ref{rem:compare} and Proposition 4.6.1 of \cite{F15}.
\end{proof}

The following simple example makes clear that the condition  that $M = M_{[d]}$
 is not equivalent to some condition on the support variety of $M$.  Conceptually, the support
variety of $M$ is the locus of 1-parameter subgroups $\psi$ at which the $p$-nilpotent action 
at $\psi$  is {\it not free}, whereas the
condition $ M  = M_{[d]}$  is the condition on the {\it triviality of the action} 
of $\cE_{B*}(v_j)$ on $M$
for all $B \in \cN_p(\fg)$ and all $j > d$.

\begin{ex}
Consider the 2-dimensional rational $\bG_a$-module $Y_R$ with basis $\{ v,w\}$ whose
$k\bG_a$-module structure is given by $u_s(w) = 0, \forall s \geq 0$; $u_s(v) = w,  s \leq R, \
u_s(v) = 0, \forall s > R$.  If $p > 2$, then $V(\bG_a)_{Y_R} = V(\bG_a) = \bA^\infty$ 
for each $R > 0$.  On the other hand, 
$Y_R$ has exponential degree $\leq R$ but does not have exponential degree $\leq R-1$ for each $R> 0$.
\end{ex}

%%%%%%%%%%%%%%%%%%%%%%%%%%%%%%%%%
%%%%%%%%%%%%%%%%%%%%%%%%%%%%%%%%%

\section{Mock Injectives, mock trivials, and the functor $(-)_{[d]}$}
\label{sec:four}

We begin this final section with a list of properties for the ``filtration by
exponential degree" functor $(-)_{[d]}:  (G\mbox{-}Mod) \to ((k[G])_{[d]}{\text -}coMod)$.
These enable us to extend the rational injectivity criterion for rational $G$-modules with
$G$ a unipotent algebraic group (given in Proposition \ref{Ucrit}) to rational 
$G$-modules for $G$ an arbitrary linear algebraic
group of exponential type.

We then briefly consider two new classes of rational $G$-modules: those
that are ``mock injective" (have trivial support varieties) and those that are
``mock trivial" (those with trivial $p$-unipotent action).     It is relatively easy
to show the existence of mock injectives and mock trivials, more difficult to
construct specific examples and study general properties of these classes. 
We conclude with a brief consideration of the right derived derived 
functors of the filtration functors $(-)_{[d]}$,
observing that they occur on the $E_2$-page of a Grothendieck spectral sequence
converging to rational cohomology.

\begin{prop}
\label{coMod}
Let $G$ be a linear algebraic group  of exponential type
and $d \geq 0$ a non-negative integer.
\begin{enumerate}
\item
The natural embedding  $\iota_d: ((k[G])_{[d]}{\text -}coMod) \ \subset \ (G{\text -}Mod)$
is exact and fully faithful. 
\item
The functor $(-)_{[d]}:  (G{\text -}Mod) \to ((k[G])_{[d]}{\text -}coMod), \ M \mapsto M_{[d]}$
is left exact and idempotent (in the sense that $(-)_{[d]} = (-)_{[d]}\circ \iota_d \circ (-)_{[d]}$).
\item
The natural embedding  $\iota_d: ((k[G])_{[d]}{\text -}coMod) \ \subset \ (G{\text -}Mod)$ 
is left adjoint to $(-)_{[d]}.$ 
\item
The category $((k[G])_{[d]}{\text -}coMod) $ has enough injectives; in other words,
for every rational $G$-module $M$ of exponential degree $\leq  d$, there exists an
inclusion of rational $G$-modules $M \hookrightarrow L$ of exponential degree $\leq d$
with $L$ an injective object of $((k[G])_{[d]}{\text -}coMod)$.
\end{enumerate}
\end{prop}

\begin{proof}
The fact that $\iota_d$ is fully faithful follows from Theorem \ref{properties}(7); the exactness
statement of Property (1) is clear.  The idempotence of Property (2) follows directly from the definition of the 
category $((k[G])_{[d]}{\text -}coMod)$ (see Theorem \ref{properties}(1)).  The left exactness
of Property (2)  is an immediate consequence of the definition of $(-)_{[d]}$.  Property (3)
is proved exactly as the adjoint property is proved in Proposition \ref{iota}.

To prove Property (4), recall that $(G{\text -}Mod)$ has enough injectives.  If $M$ is a rational
$G$-module of exponential degree $\leq  d$ and if $j: M \to I$ is an embedding of
$M$ into a rationally injective $G$-module $I$, then $j$ factors through $L \equiv I_{[d]}$
(by Theorem \ref{properties}(7)).  Since $\iota_d$ is an exact left adjoint to
$(-)_{[d]}$, $L$ is an an injective object of $((k[G])_{[d]}{\text -}coMod)$.
\end{proof}

The following necessary and sufficient criterion for rational injectivity is an extension of
Proposition \ref{Ucrit}.

\begin{prop}
\label{Gcrit}
Let $G$ be a linear algebraic group  of exponential type.
Then the following are equivalent for a rational $G$-module $L$.
\begin{enumerate}
\item
$L$ is rationally injective.
\item
For each $d \geq 0$, $L_{[d]}$ is injective in $((k[G])_{[d]}{\text-}coMod)$.
\item
For some strictly increasing sequence of non-negative integers $\{ d_i, \ i \geq 0 \}$,
$L_{[d_i]}$ is injective in $((k[G])_{[d_i]}{\text-}coMod)$ for all $i \geq 0$.
\end{enumerate}
\end{prop}

\begin{proof}
If $L$ is injective, then Proposition \ref{coMod}(4) implies that $L_{[d]} \subset L$ 
is an injective object of $((k[G])_{[d]}{\text -}coMod)$.  Namely, this is a formal 
consequence of the fact that $(-)_{[d]}$ has an exact left adjoint.  Thus, condition (1)
implies condition (2) which clearly implies condition (3).

Assume now that the rational $G$-module has the property that each
$L_{[d_i]} \subset L$ is an injective object of $((k[G])_{[d_i]}{\text -}coMod)$ for all $i \geq 0$.  
Let $M^\prime \to M$ be an inclusion of rational $G$-modules and observe that 
$(M^\prime)_{[d_i]} = M^\prime \cap M_{[d_i]}$; let $f^\prime: M^\prime \to L$ be a map of rational
$G$-modules.  Set $N_i \ = \ (M^\prime)_{[d_i]} + M_{[d_{i-1}]} \ \subset \ M_{[d_i]}$ (with $N_{-1} = 0$).
We inductively define $f_i: N_i \to L_{[d_i]}$ extending $f_{[d_i]}^\prime + f_{i-1}$ using the 
injectivity of $L_{[d_i]}$ as an object of $((k[G])_{[d_i]}{\text -}coMod)$.  Using
Theorem \ref{properties}(4), we define
$f: M \to L$ extending $f^\prime$ to be $\varinjlim_i f_i: M = \varinjlim_i N_i \to L$.
\end{proof}

\begin{defn}
\label{mockinj}
Let $G$ be a connected, linear algebraic group and $M$ a rational $G$-module.
Then $M$ is said to be {\bf mock injective} if the restriction of
$M$ to each Frobenius kernel $G_{(r)}$ is injective.

In particular, if $G$ is a linear algebraic group  of exponential 
type, a rational $G$-module $M$ is mock injective if and only if 
$V(G)_M = 0$ (by \cite[6.1]{F15}).
\end{defn}

The following proposition contrasts the behavior of injectives and mock injectives.
The first statement is merely a restatement of Theorem 4.3 and Corollary 4.5 of \cite{CPS77}.
We follow the terminology of \cite{CPS77} by saying that the algebraic group $G$ is 
reductive if and only if its connected component $G^0$ is a central product of a 
torus and a connected, semi-simple algebraic group.

\begin{prop}
\label{prop:induction}
Let $G$ be a linear algebraic group and $H$ a closed subgroup.
\begin{enumerate}
\item
The regular representation $k[G]$ of $G$ when restricted to $H$ is rationally
injective if and only if $G/H$ is an affine variety.   In particular, if $G$ is reductive,
then $k[G]$ is rationally injective as a rational $H$ module if and only if 
$H$ is reductive.
\item
$k[G]$ is always a mock injective $H$-module.
\item
If $H$ is unipotent and $M$ is a rationally injective $H$-module,
then $ind_H^G(M)$ is mock injective.
\end{enumerate}
\end{prop}

\begin{proof}
As mentioned above, (1) is a restatement of Theorem 4.3 and Corollary 4.5 of \cite{CPS77}.
The proof of (2) is given in the proof of Proposition \ref{fail} (with the notational change 
of replacing $U$ by $H$ in that proof.  Finally, (3) follows from (2) and Proposition 2.12
which asserts for $H$ unipotent that any rationally injective $H$-module $L$ is 
isomorphic to $H^0(H,L) \otimes k[H]$.
\end{proof}

\begin{cor}
\label{mockexist}
Let $G$ be linear algebraic group which is not reductive.  Then
there exist mock injective $G$-modules which are not rationally  injective.
\end{cor}

\begin{proof}
If $G$ is not reductive and $G \subset GL_N$ is a closed embedding of $G$
into some $GL_N$, then we may apply Corollary \ref{induction} to 
conclude that $k[GL_N]$ is mock injective but not injective as a rational $G$-module.
\end{proof}

\begin{prop}
Let $U$ be a connected, unipotent algebraic group.  Then there exists some rational $U$-submodule of $k[U]$
which is mock injective but not injective.
\end{prop}

\begin{proof}
Using Corollary \ref{mockexist}, choose some mock injective rational $U$-module which is 
not injective, and let 
$L_0 = H^0(U,L) \  \subset \ L$ denote the socle of $L$.  Observe that any extension
$L \to L_0\otimes k[U]$ of $L_0 \subset L_0 \otimes k[U]$ must be an inclusion.  We fix some 
such extension $i: L \to L_0\otimes k[U]$.  Write $L_0$ as a colimit of its finite dimensional subspaces, 
$$L_0 \simeq \varinjlim_{W \subset L_0, dim_k W < \infty} W, \quad  \quad 
L_0 \otimes k[U] \ = \   \varinjlim_{W \subset L_0, dim_k W < \infty} W\otimes k[U].$$
Set $L_W \ \subset \ L$ to be $i^{-1}(W \otimes k[U])$.  

Since $L$ is mock injective, $i: L \to L_0\otimes k[U]$ splits as a $kU_{(r)}$-module for all $r > 0$
so that the restriction of $i$ to any $W$, $i_W: L_W \to W \otimes k[U]$ also 
splits as a $kU_{(r)}$-module for all $r > 0$.   Since $L$ is not injective, $i$ is not an isomorphism
so that some $i_W$ is not an isomorphism.  We may therefore assume that $L = L_W$; in which
case, $H^0(U,L) = W$ is finite dimensional.   

We  now proceed by descending 
induction on the dimension of $L_0$.  
Let $L_0^\prime \subset L_0$ be a subspace of codimension 1 with 1-dimensional quotient $\pi: L \to \ol L_0$;
set  $\ol L$ to be the image of $L$ under the map $(\pi\otimes 1) \circ i: L \to k[U]$, and set $L^\prime \equiv ker\{ L \to \ol L \}$.  
We employ the following map of short exact sequences of rational $U$-modules:
\begin{equation}
\label{induction}
\xymatrix{0 \ar[r] & L^\prime \ar[r] \ar[d]_{i\prime}  & L \ar[r] \ar[d]_i  & \ol L \ar[r] \ar[d]_{\ol i} \ar[r] & 0\\
0 \ar[r] & L_0^\prime \otimes k[U] \ar[r] & L_0 \otimes k[U]  \ar[r]^-{\pi\otimes 1} & \ol L_0 \otimes k[U] \ar[r] & 0} .
\end{equation}

Since $L$ is injective as a $kU_{(r)}$-module, the inclusion $i: L \to L_0\otimes k[U]$ 
splits as a map of $kU_{(r)}$-modules for any $r > 0$.  Together with a (vector space) splitting
of $\pi: L_0 \to \ol L_0$,  a $kU_{(r)}$-splitting of $i$ determines a $kU_{(r)}$-splitting
of $\ol L \to \ol L_0 \otimes k[T]$.  Thus, $\ol L$ is free as a $kU_{(r)}$-module for any $r > 0$.  Consequently, the 
top row of (\ref{induction}) is split as a short exact sequence of $kU_{(r)}$-modules for any $r > 0$.  Thus, 
$L^\prime$ is mock injective.  

Finally, either $\ol i$ is not an isomorphism (providing a non-injective, mock injective submodule
of $k[U]$), or we may apply induction since $dim((L^\prime)_0) < dim(L)$.
\end{proof}

We next list a few closure properties of the category of mock injectives.

\begin{prop}
Let $G$   be a connected, linear algebraic group of exponential type.    
\begin{enumerate}
\item 
If $M_1, M_2$ are rational $G$-modules which are mock injective, then $M_1 \otimes M_2$ is 
mock injective.
\item
If $M_1, M_2$ are rational $G$-modules which are mock injective and
if 
$$0 \to \ M_1 \ \to \ M \to \ M_3 \to 0$$ is a short exact sequence of rational $G$-modules,
then $M$ is also mock injective.
\item
An arbitrary direct sum of mock injective rational $G$-modules $\{ M_i, \ i\in I\}$, 
$\bigoplus_{i \in I} M_i$, is mock injective.
\end{enumerate}
\end{prop}

\begin{proof}
We use the criterion for  mock injectivity mentioned in Definition \ref{mockinj}.  With this criterion,
(1) and (2) follow immediately from Theorem 4.6 of \cite{F15}, as does property (3) if $I$ is finite.
To prove (3) for any arbitrary indexing set $I$, we use the observation that $\ul B \notin V(G)_M$
it suffices to show that the action of $\sum_s \cE_{B_s *}(u_s)$ determines a free action of
$k[t]/t^p$ on $M$; this condition is clearly inherited by arbitrary direct sums.
\end{proof}

We now introduce the class of mock trivial $G$-modules.

\begin{defn}
Let $G$ be a linear algebraic group with a structure of exponential type.
Then a rational $G$-module $M$ is said
to be {\bf mock trivial} if $M = M_{0]}$; equivalently, if the coproduct structure
$\Delta: M \to M\otimes k[G]$ factors through $M \to M \otimes k[G]_{[0]}$.  
\end{defn}

\begin{prop}
Let $G$ be a linear algebraic group with a structure of exponential type.
Then a rational $G$-module $M$ is  {\bf mock trivial} if and only if the pull-back 
of $M$ along any 1-parameter subgroup $\psi: \bG_a \to G$ is trivial.

This implies in particular that the $p$-nilpotent action of $G$ at every 1-parameter
subgroup $\psi \in V(G)$ is trivial which in turn implies that the support variety
of $M$, $V(G)_M$, equals all of $V(G)$.
 \end{prop}

\begin{proof} 
By Proposition \ref{alternate},  $M$ is mock trivial if and only if $\cE_B^*(M)$ is
trivial as $\bG_a$-module for all 1-parameter subgroups of the form 
$\cE_B: \bG_a \to G$, and this is the case if and 
only if $(\cE_B \circ F^i)^*(M)$ is trivial as a $\bG_a$-module for all $\cE_B: \bG_a \to B$
and all $i \geq 0$.  Since any 1-parameter subgroup of $G$ is of the 
form $\cE_{\ul B} = \prod_s \cE_{B_s} \circ F^s$ and since the 
action of $(\prod_s \cE_{B_s }\circ F^s)_*(v_j)$ on $M$ equals the product of
the actions of $(\cE_{B_s})_*(v_{j-p^s})$ on $M$,  we conclude that $\psi^*(M)$ is a 
trivial $\bG_a$-module whenever $M$ is mock trivial.  The converse is clear
from the first equivalence mentioned at the beginning of this proof.

As recalled prior to Definition \ref{ratsupp}, the $p$-nilpotent action of $G$ 
at $\cE_{\ul B} = \prod_s \cE_{B_s} \circ F^s$  is defined in \cite[2.9.1]{F15}
to be the action of $(\sum_{s \geq 0} \cE_{B_s})_*(u_s)$.  By the preceding paragraph, if 
$M$ is mock trivial, this action is trivial.  This immediately implies that $V(G)_M = V(G)$
for any mock trivial rational $G$-module.
\end{proof}

We state a few properties of the class of mock trivial $G$-modules.  Of course,
even the class of trivial $G$-modules need not be closed under extensions.  

\begin{prop}
Let $G$ be a linear algebraic group with a structure of exponential type.  
\begin{enumerate}
\item
If $G \ \not= \ U_p(G)$, then $k[G]_{[0]}$ is a non-trivial sub Hopf algebra of $k[G]$
and thus is an indecomposable mock trivial module, but not trivial.  
\item
If $G \ \not= \ U_p(G)$, then there exist finite dimensional rational $G$-modules
which are mock trivial but not trivial.
\item
If $H \subset G$ is an embedding of exponential type such that $x^p = 1, \ x \in H$
(i.e., $H = U_p(H)$ and if $M$ is
a mock trivial rational $G$-module, then  $M$ restricted to $H$ is trivial.
\item
If $M$ is a rational $G$-module such that $M = M_{[0]}$ and if $M^\prime \subset M$
is rational $G$-submodule, that $M^\prime = (M^\prime)_{[0]}$; similarly, any quotient
$\overline M$ of a rational $G$-module such that $M = M_{[0]}$ also satisfies 
$\ol M = (\ol M)_{[0]}$.
\item
A colimit $\varinjlim_\alpha L_\alpha$ of mock trivial $G$-modules  
$L_\alpha$ is again mock trivial.
\end{enumerate}
\end{prop}

\begin{proof}
By Proposition \ref{prop:embed}, $k[G]_{[0]}$ is non-zero whenever $G \ \not= \ U_p(G)$
and by Proposition \ref{subcoalg} $k[G]_{[0]}$ is a sub Hopf algebra of $k[G]$; in particular, $k[G]_{[0]}$ is
a non-trivial rational $G$-module $M$ such that $M = M_{[0]}$.  Since the socle of $k[G]_{[0]}$ is 1 dimensional,
$k[G]_{[0]}$ is indecomposable.  To prove (2), we apply (1) and recall that any rational $G$-module is
a union of its finite dimensional submodules.

If $H \subset G$ is an embedding of exponential type, then $\cE: \cN_p(G) \times \bG_a \to G$
restricts to $\cE_H: \cN_p(H) \times \bG_a \to H$;.  this implies that if the comodule structure
for $M$ has the property that it arises from a coproduct $M \to M \otimes k[G]_{[0]}$, then the restriction
to $H$ has the property that the coproduct arises from a coproduct $M  \to M \otimes k[H]_{[0]}$.
Thus, Property (3) follows from Proposition \ref{alternate} (in the special case $d=1$).

Properties (4) and (5) are evident properties of the abelian category of $k[G]_{[0]}$ comodules.
\end{proof}

In the following proposition, we view (closed) points of the support variety
of a finite group scheme as equivalence classes of $p$-points as in \cite{FP1}
rather than use 1-parameter subgroups which provide distinguished representatives
of equivalence classes of $\p$-points as in \cite{SFB1}.  Since we do not
consider the scheme structure of support varieties, we do not use the language and
technology of $\pi$-points found in \cite{FP2}.  For a finite group scheme
$H$, $P(H)$ consists of the closed points of the scheme $\Pi(H)$ of $\pi$-points;
the points of $P(H)$ are equivalence classes of $p$-points of $H$.

\begin{prop}
\label{characterize}
Let $G$ be a linear algebraic group of exponential type.  If a rational $G$-module $M$
is mock trivial then the restriction of $M$ to each Frobenius kernel $G_{(r)}$ satisfies
the condition that every for every $[\alpha] \in P(G_{(r)})$ there exists a representative 
$\alpha: k[t]/t^p \to kG_{(r)}$ such that $\alpha^*(M)$ is trivial as a $k[t]/t^p$-module.
\end{prop}

\begin{proof}
There is a natural homeomorphism relating $P(G_{(r)})$ to $V_r(G_{(r)}) \simeq \cC_r(\cN_p(\fg))$
given by sending $\ul B \in \cC_r(\cN_p(\fg))$ to  the $p$-point 
$\alpha_{\ul B}: k[t]/t^p \to kG_{(r)}, \ t\mapsto \cE_{(B_0,\ldots,B_{r-1})}(u_{r-1}).$
As shown in \cite[4.3]{F15}, the equivalence class of $[\alpha_{\ul B}] \in P(G_{(r)})$ 
is also represented by  the $\pi$-point $k[t]/t^p \to kG_{(r)}, \ t\mapsto \sum_{i=0}^{r-1} (\cE_{B_i})_*(u_{r-1-i})$
(see the discussion after Definition \ref{ratsupp}).  If $M$ is mock injective, then 
each $\sum_{i=0}^{r-1} (\cE_{B_i})_*(u_{r-1-i})$ acts trivially on $M$.
\end{proof}

We conclude with Grothendieck spectral sequences relating rational cohomology
to the structures we have considered.  We view the $t$-th right derived functor 
of the left exact functor $(-)_{[d]}$, 
$$R^t(-)_{[d]}: (G{\text -}Mod) \ \to \ ((k[G])_{[d]}{\text -}coMod),$$
as ``derived filtrations functors".

\begin{prop}
\label{Groth}
Let $G$ be a linear algebraic group  of exponential type.
For any $d \geq 0$, there is a natural identification of functors
$$H^0(G,-) \ \simeq \ Hom_{(k[G])_{[d]}{\text -}coMod)}(k,-) \circ (-)_{[d]}: (G\text{-}Mod) \to (vsp/k)$$
leading to a spectral sequence
$$R^sHom_{(k[G])_{[d]}{\text -}coMod)}(k,-) \circ R^t(-)_{[d]}(M) \ \Rightarrow H^{s+t}(G,M).$$
\end{prop}

\begin{proof}
The asserted identification of the composition  $Hom_{(k[G])_{[d]}{\text -}coMod)}(k,-) \circ (-)_{[d]}$
with $H^0(G,-)$ is made by observing that both send a rational $G$-module $M$ to the
subspace $M^G$ of invariant elements (which consists of those $m \in M$ such that 
$\Delta_M(m) = m\otimes 1 \in M \otimes k[G]$).

Since the functor $(-)_{[d]}$ has an exact left adjoint by Proposition \ref{coMod}(3) and therefore 
sends injectives to injectives,  
the Grothendieck spectral sequence for a composition of left exact functors applies and
has the asserted form (see \cite{Wei}).
\end{proof}

\end{document}